\documentclass[a4paper,11pt]{amsart}
\usepackage{mathtext}         %

\usepackage{latexsym}
\usepackage{amsmath,amssymb,amsthm,amsopn,amscd}
\usepackage[dvips]{graphicx}
\usepackage{array,fullpage,wrapfig}
\usepackage{euscript}
\usepackage[matrix,arrow,curve]{xy}
\newcommand{\rf}[1]{(\ref{#1})}

\newtheorem{theorem}{Theorem}[section]
\newtheorem*{theorem*}{Theorem}
\newtheorem{proposition}[theorem]{Proposition}
\newtheorem*{proposition*}{Proposition}
\newtheorem{statement}[theorem]{Statement}
\newtheorem{lemma}[theorem]{Lemma}
\newtheorem{corollary}[theorem]{Corollary}
\newtheorem*{corollary*}{Corollary}

\theoremstyle{definition}
\newtheorem{example}[theorem]{Example}
\newtheorem{definition}[theorem]{Definition}

\theoremstyle{remark}
\newtheorem{remark}[theorem]{Remark}

\newcommand{\ldot}{{\:\raisebox{1.5pt}{\selectfont\text{\circle*{1.5}}}}}
\newcommand{\udot}{{\:\raisebox{4pt}{\selectfont\text{\circle*{1.5}}}}}

\def\kk{\Bbbk}

\let\le\leqslant
\let\ge\geqslant
\let\leq\leqslant
\let\geq\geqslant
\let\cong\simeq

\newcommand{\Z}{\mathbb Z}
\newcommand{\C}{\mathbb C}

\def\d{\partial}
\def\CP1{\mathbb{C}\mathrm{P}^1}

\def\oM{\overline{\mathcal{M}}}
\def\mm{\mathsf{m}}
\def\mmg{\mathsf{\bar{m}}}

\def\BV{\mathsf{BV}}
\def\Hycomm{\mathsf{Hycomm}}
\def\Gerst{\mathsf{Gerst}}
\def\Grav{\mathsf{Grav}}

\def\P{\EuScript{P}}
\def\Q{\EuScript{Q}}

\def\End{{\mathrm{End}}}
\def\Bar{\EuScript{B}}
\def\dual{{\vee}}

\def\deg{{\mathrm{deg}}}
\def\Top{{\mathcal{T}op}}

\def\tt{\text{-}}

\title{Hypercommutative operad as a homotopy quotient of BV}

\author{A.~Khoroshkin}

\address{A.~Khoroshkin:\newline
Simons Center for Geometry and Physics,
State University of New York\newline
Stony Brook, NY 11794-3636,
U.S.A.
\newline
and 
ITEP, Bolshaya Cheremushkinskaya 25, 117259, Moscow, Russia
}

\email{anton.khoroshkin@scgp.stonybrook.edu}

\author{N.~Markarian}

\address{N.~Markarian:\newline
Department of Mathematics,
National Research University Higher
School of Economics, \newline
Ul. Vavilova 7, Moscow 117312, Russia}

\email{nikita.markarian@gmail.com}

\author{S. Shadrin}

\address{S.~Shadrin:\newline
Korteweg-de Vries Institute for Mathematics,
University of Amsterdam, \newline
P.~O.~Box 94248, 1090 GE Amsterdam, 
The Netherlands}
\email{s.shadrin@uva.nl}

\thanks{
The first author's research was partially supported by the grants
 NSh-3349.2012.2,
RFBR-10-01-00836,
RFBR-CNRS-10-01-93111, RFBR-CNRS-10-01-93113,
and the grant
by Ministry of Education and Science of the Russian Federation under
contract 14.740.11.081.
The second author's research was partially supported by
AG Laboratory GU-HSE, RF
government grant, ag.~11.G34.31.0023 and the National Research University
Higher School of Economics' Academic Fund Program in 2012-2013,
research grant No.~11-01-0145, RFBR--12-01-00944.
The third author's research was supported by a Vidi grant of the Netherlands Organization for Scientific Research.
}

\begin{document}

\begin{abstract}
We give an explicit formula for a quasi-isomorphism between
the operads $\Hycomm$ (the homology of the moduli space of stable genus $0$ curves)
and $\BV/\Delta$ (the homotopy quotient of Batalin-Vilkovisky operad by the $\BV$-operator).
In other words we derive an equivalence of $\Hycomm$-algebras and 
$\BV$-algebras enhanced with a homotopy that trivializes the $\BV$-operator.

These formulas are given in terms of the Givental graphs, and are proved in two different ways. 
One proof uses the Givental group action, and the other proof goes through a chain of explicit formulas on resolutions of $\Hycomm$ and $\BV$. The second approach gives, in particular, a homological explanation of the Givental group action on $\Hycomm$-algebras.
\end{abstract}

\maketitle


\setcounter{section}{-1}

\section{Introduction}
The main purpose of this paper is to describe a natural equivalence 
between the category of differential graded Batalin-Vilkovisky algebras enhanced with the trivialization of BV-operator 
and the category of formal Frobenius manifolds without a pairing (also known under the name of hypercommutative algebras).
The problem we are discussing has an explicit topological origin.
I.~e., we are looking for an  equivalence of the operad of moduli spaces of stable curves
and a homotopy quotient of the framed discs operad by the circle action.
Having in mind that both topological operads under consideration are known to be formal
we restrict ourselves to the corresponding relationship of the homology operads.
We suggest a pure algebraic solution of the problem accompanied with an exact formula for the desired quasi-isomorphism.

Let us first briefly recall the definitions of two categories under consideration using the language of operads.
Consider the moduli spaces of stable genus $0$ curves $\oM_{0,n+1}$, $n=2,3,\dots$. 
A stable genus $0$ curve is a nodal curve of arithmetic genus $0$ with $(n+1)$ pairwise 
distinct marked points in its smooth part, and it has at least three special points 
(nodes or marked points) on each of the irreducible components. 
The points are labeled by the numbers $0,1,\dots,n$. 
There is a natural stratification by the topological types of nodal curves. 
The strata of codimension one can be realized as the images of the gluing morphism 
$\rho=\rho_i\colon\oM_{0,n_1+1}\times\oM_{0,n_2+1}\to\oM_{0,n+1}$, $n=n_1+n_2-1$, $i=1,\dots,n_1$,
 where the new nodal curve is obtained by attaching the zero point of a curve in $\oM_{0,n_2+1}$ 
to the $i$-th point of a curve in $\oM_{0,n_1+1}$. These morphisms define on the spaces $\oM_{0,n+1}$, $n=2,3,\dots$, the structure of a topological operad. Therefore, the homology of the spaces $\oM_{0,n+1}$, $n=2,3,\dots$, are endowed with the structure of an algebraic operad. 
This operad is called \emph{the hypercommutative operad} and we denote it by $\Hycomm$. 
We recall an explicit description of the \emph{hypercommutative algebra} in 
Section~\ref{sec::open_close}. 
We refer to~\cite{Man} 
for details and to~\cite{Keel}
for the description of the intersection theory on $\oM_{0,n+1}$.
(Note that Manin uses in~\cite{Man} the notation $\mathcal{C}om_\infty$ for the operad of hypercommutative algebras.)

Another important topological operad under consideration is the framed little discs operad.
The set of $n$-ary operations of this operad consists of configurations of the disjoint union of $n$ small discs 
inside the unit disc, each inner disc has a marked point on the boundary.
It is equivalent to mark a point on the boundary of the circle or to fix a rotation of the inner disc which gives an identification of the inner disc with a standard disc of the same radius.
The gluing of the outer boundary of the unit disc coming from configuration of $n_1$ small discs
with the boundary of the $i$'th inner disc of the configuration of $n_2$ small discs defines a configuration of $n_1+n_2+1$ small
pointed discs which prescribes the composition rules in the operad. 
The homology of this operad is known under the name of Batalin-Vilkovisky operad and has very simple description 
in terms of generators and relations.
Namely, a (differential graded) Batalin-Vilkovisky algebra is a graded commutative associative algebra with two operators, 
$d$ of degree $1$ and $\Delta$ of degree $-1$, such that $d^2$, 
$\Delta^2$, and $d\Delta+\Delta d$ are equal to zero, 
$d$ is a derivation and $\Delta$ 
is a differential operator of the second order with respect to the multiplication. 

These two algebraic structures, hypercommutative algebras and Batalin-Vilkovisky algebras are known to be closely related. 
The hypercommutative algebra structure is the most important ingredient of a formal Frobenius manifold structure.
A typical application of a relation between $\BV$-algebras and hypercommutative algebras is that under some conditions 
like Hodge property or some trivialization of the $\BV$-operator $\Delta$ we obtain a Frobenius manifold structure 
on the cohomology of a $\BV$-algebra; we refer to~\cite{BerCecOogVaf,BarKon,LosSha,Katzarkov_Pantev,DotKho,DruVal,DotShaVal,Dru} 
for different aspects and different examples of this kind of correspondence and relations between them. 
The topological origin of all these statements looks as follows.
The homotopy quotient of the framed little discs operad by rotations is weakly equivalent to the operad of moduli spaces
of stable genus $0$ curves. This statement was mentioned in~\cite{Mar} and written in details in~\cite{Dru}.
We are focused on the algebraic counterpart of this statement equipped with precise formulas.

In topology the homotopy quotient functor by the group $G$ is a functor from the category of
$G$-spaces to the category of spaces which is defined as a left adjoint functor to the trivial embedding:
any topological space admits a trivial action of the group $G$.
The algebra over the homotopy quotient by $G$ of a given operad $\P$
is an algebra over $\P$ where the action of $G$ is trivialized.
We will show the equivalence of this two definitions 
in the particular case $G=S^1$ and $\P=\BV$.

In general, the condition on trivialization of the $\BV$-operator $\Delta$ that one has to use can be formulated 
in several different ways. First, we require that $\Delta$ is homotopically trivial, 
that is, the full homotopy transfer of $\Delta$ on the cohomology of $d$ is equal to zero. 
Equivalently, we can say that the spectral sequence (if exists) for $(d,\Delta)$ converges on the first page.
(See~\cite{DotShaVal} for details of this approach.)
We use the different but similar approach.
Consider the bi-complex $V[[z]]$ with differential $d+z\Delta$, where $z$ is a formal parameter of homological degree $2$,
and consider a particular trivialization (homotopy) for the action of $\Delta$.
Namely, we choose a particular automorphism of the space $V[[z]]$ which gives a quasi-isomorphism of
complexes with respect to differentials $d$ and $d+z\Delta$.
The other possible way to say the same is that 
$d+z\Delta=\exp(-\phi(z))d\exp(\phi(z))$, where $z$ is a formal variable, 
and $\phi(z)=\sum_{i=1}^\infty \phi_iz^i$ is some series of operators.
 We consider these extra operators $\phi_i$, $i=1,2,\dots$, as a part of the algebraic structure we have, 
and a $\BV$ algebra equipped with this extra trivialization data is a representation of the homotopy quotient of the $\BV$ operad. 
We denote this model of the homotopy quotient by $\BV/\Delta$.  

The main result of this paper is an explicit formula for a quasi-isomorphism $\theta\colon\Hycomm\to\BV/\Delta$.
This result summarizes the relations between hypercommutative algebras and Batalin-Vilkovisky algebras mentioned above.
The equivalence of the homotopy categories of $\Hycomm$-algebras 
and homotopy quotient of $\BV$-algebras was given in~\cite{Dru} on the level of chains.

There are two ways to construct this map: 

First approach goes through a careful analysis of a system of relations between the operads $\Hycomm$, $\BV/\Delta$,
the operad of Gerstenhaber algebras and the gravity operad.
It deals with different precise relationships between homotopy quotients and equivariant (co)homology
first discovered by Getzler in~\cite{Getzler_grav,Get,Getzler_genus0}.
{\bf Theorem~\ref{thm::diag:Hycom->BV}} summarizes these relationships in main {\bf Diagram~\eqref{eq::diag::Hycom_BV} }
of quasi-iso relating $\BV/\Delta$ and $\Hycomm$. 
We go through Diagram~\eqref{eq::diag::Hycom_BV} specifying the generating cocycles in the cohomology at each step. 
As a result we get a formula for $\theta$ given in terms of summations over three-valent graphs.

Second approach is a generalization of the interpretation of the BCOV theory suggested in~\cite{Sha}. 
There is an action of the loop group of the general linear group on the representations of $\Hycomm$ in a given vector space. 
It was constructed by Givental and the action of its Lie algebra was studied by Y.-P.~Lee, see~\cite{Giv3,Lee1}. 
We generalize this group action to an action on the space of morphisms from $\Hycomm$ to an arbitrary operad. 
This way we can describe the map $\theta$ as an application of a particular Givental group element to a very simple morphism 
from $\Hycomm$ to $\BV/\Delta$, the one that preserves the commutative associative product and ignores all the rest.
In this case the final formula is given in terms of summations over graphs with arbitrary valencies of vertices.

We state that these two formulas for $\theta$ coincide, however, we prefer to omit the direct proof of this statement
and use the uniqueness arguments in order to explain the coincidence.
The Givental-style formula is simpler for applications and contains already all cancellations, however,
the homological approach is of it's own interest. 
In particular, it allows to give additional point of view on the $\psi$-classes which we want to use elsewhere.
So far, we show how one can get the topological recursion relations using this homological interpretation.

Finally, our result on an explicit quasi-isomorphism formula allows to give a new interpretation for 
the Givental group action mentioned above. It appears that the action of the Givental group on morphisms 
of $\Hycomm$ corresponds to the ambiguity of a particular choice of a trivialization for $\Delta$ in $\BV/\Delta$.

\subsection{Outline of the paper}
We repeat once again that in spite of topological motivation all proofs and all expositions are purely algebraic.
All operads involved and algebras over them are defined in pure algebraic terms of generators and relations.

In Section~\ref{sec:explicitformula} we formulate our main result on an explicit formula for the quasi-iso $\theta\colon\Hycomm\to\BV/\Delta$. 
Section~\ref{sec::circle} deals with the circle action. 
Namely, the categorical definition of the homotopy quotient by $\Delta$ is given in~\ref{sec::homotopy_quotients_def} and 
the algebraic counterpart of Chern classes is presented in~\ref{sec::chern}.

In Section~\ref{sec::operad::definitions} we introduce notations and definitions 
for all operads involved in the main chain of quasi-isomorphisms between $\Hycomm$ and $\BV/\Delta$ 
(Diagram~\eqref{eq::diag::Hycom_BV}). 
This part is quite technical and is needed mainly to fix the notation.

Section~\ref{sec::main_diagram} contains the main Diagram~\eqref{eq::diag::Hycom_BV} of quasi-iso connecting $\Hycomm$ and $\BV/\Delta$.
We play around it in order to get algebraic description of $\psi$-classes and a useful dg-model of the $\BV$-operad.
In Section~\ref{sec:diagrammatic} we go through all these quasi-isomorphisms specifying generating cocycles 
in the cohomology, 
and this way we obtain a direct map $\theta\colon\Hycomm\to\BV/\Delta$. 
In Section~\ref{sec:givental} we recall the Givental theory, apply it in order to get a formula for 
$\theta$ from Section~\ref{sec:explicitformula},
 and then use the existence of such a map in order to give a new interpretation for the Givental theory.

Those readers who are interested more in the results rather than in the proofs may skip technical 
Sections~\ref{sec::operad::definitions} and \ref{sec:diagrammatic}.

\subsection{Acknowledgment}
We are grateful to V.~Dotsenko, E.~Getzler, A.~Givental, G.~Felder, A.~Losev, and B.~Vallette 
for useful discussions on closely related topics.


\section{An explicit formula}
\label{sec:explicitformula}
In this section we give an explicit formula for a map $\Hycomm\to \BV/\Delta$ that takes $\Hycomm$ isomorphically to the cohomology of $\BV/\Delta$.

\subsection{A presentation of $\BV/\Delta$}
\label{sec::bv/delta}
The definition of a homotopy quotient given below is more convenient for applications 
then the standard categorical definition.
We discuss the equivalence of these definitions in Section~\ref{sec::homotopy_quotients_def}.

The algebras over the homotopy quotient $\BV/\Delta$ 
 are in one-to-one correspondence with 
the $\BV$-algebras where $\Delta$ acts trivially on homology and moreover one chooses 
a particular trivialization for this action.
I.~e. the $\BV/\Delta$ algebra on a complex $(V^{\udot},d)$ consists of 
commutative multiplication, differential operator $\Delta: V^{\udot}\rightarrow V^{\udot}[-1]$ of order at most $2$ and 
an isomorphism of complexes
\begin{equation}
\label{eq::homotopy_quotients_def}
\Phi(z): (V^{\udot}[[z]],d + z\Delta) \rightarrow (V^{\udot}[[z]],d),
\end{equation}
where $z$ is a formal parameter of degree $2$ and $\Phi(z)$ is a formal power series in $z$.
I.~e. $\Phi(z) = \sum_{i\geq 0} \Phi_i z^i$. $\Phi_i$ should be linear endomorphisms of the vector space $V^{\udot}$
 of pure homological degree $-2i$ and $\Phi_0=Id_{V}$.

Our formulas below become simpler if we consider exponential coordinates for trivialization.
Namely, we represent $\Phi(z)$ as a series $\exp( \phi(z))$, $ \phi(z):=\sum_{i\geq 1} \phi_i z^i$ that is, 
\begin{equation*}
Id_{V} + \Phi_1 z + \Phi_2 z^2 +\ldots = \exp( \phi_1 z +\phi_2 z^2 +\ldots ).
\end{equation*}
This allows us to describe the operad $\BV/\Delta$ in the following way. 

In order to homotopically resolve the operation $\Delta$ in the operad $\BV$ 
we have to add a number of generators $\phi_i$, $i\geq 1$, 
$\deg \phi_i=-2i$, and define a differential $d$ 
that vanishes on all generators of $\BV$ operad and such that $\Delta$ 
itself becomes an exact cocycle, while the rest of the 
$\BV$-structure survives in the cohomology (and no new cohomology cycles appear).
We rewrite the formula $d\exp(\phi(z))=\exp(\phi(z)) (d+z\Delta)$ as 
\begin{equation}\label{eq:d-delta}
\Phi(z)^{-1}d\Phi(z) = \exp(-\phi(z))d\exp(\phi(z))=d+z\Delta
\end{equation} 
and use the expansion of the left hand side of the latter equation in order to define a differential that we denote by $\Delta\frac{\d}{\d\phi}$ on the generators $\phi_i$ 
as an expression for $[d,\phi_i]$. That is, the formulas
\begin{align}\label{eq:d-a-0}
\Delta & = [d,\phi_1], \\ \notag
0 & = [d,\phi_2]+\frac{1}{2} [[d,\phi_1],\phi_1],\\ \notag
0 & = [d,\phi_3]+ [[d,\phi_1],\phi_2]+ [[d,\phi_2],\phi_1] + \frac{1}{6}[[[d,\phi_1],\phi_1],\phi_1], 
\end{align}
turn into
\begin{align}\label{eq:d-a}
\Delta\frac{\d}{\d\phi}(\phi_1) & = [d,\phi_1] = \Delta \\ \notag
\Delta\frac{\d}{\d\phi}(\phi_2) & = [d,\phi_2] = -\frac{1}{2}[\Delta,\phi_1] \\ \notag
\Delta\frac{\d}{\d\phi}(\phi_3) & = [d,\phi_3] =-[\Delta,\phi_2] +\frac{1}{3}[[\Delta,\phi_1],\phi_1],  
\end{align}
respectively. 

We define the operad $\BV/\Delta$ to be the operad obtained by adding to $\BV$ 
the generators $\phi_i$, $i\geq 0$, with the differential $\Delta\frac{\d}{\d\phi}$  equal to zero on $\BV$ and 
given by Equations~\eqref{eq:d-a}. We use the notation $\Delta\frac{\d}{\d\phi}$ for the differential in order to point out
that it decreases the degree in $\phi$ by $1$ and increases the degree in $\Delta$ also by $1$ so looks like a differential operator 
$\Delta\frac{\d}{\d\phi}$.

\subsection{A formula for quasi-isomorphism}\label{sec:quasi}

We construct a map $\Hycomm\to\BV/\Delta$. To the generator $\mm_n\in \Hycomm(n)$ given by the fundamental cycle $[\oM_{0,n+1}]$ 
we associate an element $\theta_n$ of $\BV/\Delta(n)$ represented as a sum over all possible rooted trees with $n$ leaves, where 
\begin{itemize}
\item at the each vertex with $k$ inputs we put the $(k-1)$ times iterated product in the $\BV$-algebra.
The iterated product $m(x_1,\dots,x_k)$ is defined as $m(x_1,\dots m(x_{k-2},m(x_{k-1},x_k))\dots)$,
where $m(x_1,x_2)$ denotes the usual binary multiplication from $\BV(2)$.
Abusing the notation we denote the iterated product by the same letter $m$.
\item
Each input/output $e$ of any given vertex in a graph is enhanced by a formal parameter $\psi_{e}$.
I.~e. a vertex with $k$ inputs will be equipped with $k+1$ additional parameters.
These parameters will be used to determine the combinatorial coefficient of the graph.
\item On each leaf $e$ (an input of the graph) we put the operator $\exp(-\phi(-\psi_e))$, 
where $\psi_e$ is the defined above formal parameter associated 
to the corresponding input $e$ of the vertex where the leaf is attached. 
\item At the root (the output of the graph) we put the operator $\exp(\phi(\psi))$.
 Again, $\psi$ is a formal parameter associated to the output of the vertex, where the root is attached.
\item At the internal edge that serves as the output of a vertex $v'$ and an input of a vertex $v''$ we put the operator
\begin{equation*}
\mathcal{E}:=-\frac{\exp(-\phi(-\psi'')\exp(\phi(\psi'))-1}{\psi''+\psi'},
\end{equation*}
where $\psi'$ (respectively, $\psi''$) are attached to the output of $v'$ 
(respectively, the corresponding input of $v''$) in the same way as above.
\end{itemize}
Each graph should be considered as a sum of graphs obtained by expansion 
of all involved series in $\psi$'s, and each summand has a combinatorial coefficient 
equal to the product over all vertices of the integrals 
\begin{equation}\label{eq:psi-int}
\int_{\oM_{0,k+1}}\psi_0^{d_0}\psi_1^{d_1}\cdots \psi_k^{d_k} := \begin{cases}
\frac{(k-2)!}{d_0!d_1!\cdots d_k!}, & \mbox{if } k-2=d_0+d_1+\cdots+d_k; \\
0, & \mbox{otherwise, }
\end{cases} 
\end{equation}
where the degrees $d_0,d_1,\dots,d_k$ are precisely the degrees of $\psi$-classes 
associated to the inputs/output of a vertex. 

Note that after expansion of all exponents there are only finitely many monomials in $\psi$'s 
 that contribute in the 
total summand for $\theta_n$.
Consequently, $\theta_n$ is represented by a finite sum of combinations of multiplications and $\phi_i$'s.
 In particular, the total degree of each nonzero term is equal to $2-2n$
(recall that $\deg m =0$ and $\deg \phi_i=-2i$).

\begin{remark}
 Here $\psi$-classes and their integrals over the space $\oM_{0,k+1}$,
 as in Equation~\eqref{eq:psi-int}, should be understood as a formal notation for 
some combinatorial constants (multinomial coefficients).
 However, in Sections~\ref{sec:givental} and~\ref{sec::TRR} we clarify the geometric meaning and 
the origin of this formula.
\end{remark}

\begin{example}
\label{ex::theta_23}
 Explicit formulas for the $\theta_2$ and $\theta_3$.
\begin{align*}
\theta_2\left(x_1,x_2\right) =  & m\left(x_1,x_2\right) \\ 
\theta_3\left(x_1,x_2,x_3\right) =  & \phi_1\left(m\left(x_1,x_2,x_3\right)\right) 
+ \left( m\left(x_1,x_2,\phi_1(x_3)\right) + m\left(x_2,x_3,\phi_1(x_1)\right) + m\left(x_3,x_1,\phi_1(x_2)\right)\right)\\ 
& -\left(m\left(x_1,\phi_1\left(m(x_2,x_3)\right)\right) + m\left(x_2, \phi_1\left(m(x_1,x_3)\right)\right)
+m\left(x_3, \phi_1\left(m(x_1,x_2)\right)\right)\right).
\end{align*}
\end{example}

\begin{theorem}
 \label{thm::formula_BV-Hycomm}
Using the Leibniz rule, the map $\theta$ defined on generators by $\theta\colon \mm_n\mapsto \theta_n$ extends to a morphism of operads $\theta\colon\Hycomm\rightarrow \BV/\Delta$. Moreover, $\theta$ is a quasi-isomorphism of operads.
\end{theorem}

We present two ways to prove this theorem.

The first proof uses computations with equivariant homology. It is presented in Section~\ref{sec:diagrammatic}. 
First, we give a sequence of natural quasi-iso connecting $\Hycomm$ and $\BV/\Delta$.
Second, a careful diagram chase allows us to obtain a formula for $\theta$, 
and, in addition, a natural homological explanation of the 
Givental group action on representations of $\Hycomm$.

The second proof also consists from two steps.
The first step is the same. We observe that the cohomology of $\BV/\Delta$ coincides with $\Hycomm$.
Second, we notice that the expression for $\theta_k$ does not contain $\Delta$ and 
therefore $\theta_k\notin Im(\Delta\frac{\d}{\d\phi})$.
Third, using a certain generalization of the Givental theory we show that $\theta_k$ are $\Delta\frac{\d}{\d\phi}$-closed.
The degree count implies that $\theta$ defines a quasi-isomorphism of operads.
This proof is explained in detail in Section~\ref{sec:givental}.

\subsection{Examples} There are natural examples of the $\BV/\Delta$ algebras structures on the de Rham complexes of Poisson and Jacobi manifolds. These examples are discussed in detail in~\cite{DotShaVal} from a different perspective. 

In the case of a Poisson manifold, we consider its de Rham complex with the de Rham differential $d^{dR}$ and wedge product, and the operator $\phi_1$ equal to the contraction with the Poisson structure and $\phi_i=0$, $i=2,3,\dots$. The operator $\Delta=[d^{dR},\phi_1]$ is a $\BV$-operator, thus we have a natural structure of a $\BV/\Delta$ algebra. In the case of Jacobi manifold the $\BV/\Delta$ structure exists on the space of basic differential forms, the construction is very similar, and we refer to~\cite{DotShaVal} for details.

In both cases, the explicit formulas for $\theta_k$, $k=2,3,\dots$, gives a structure of the $\Hycomm$ algebra on the cohomology in these examples. In fact, with these formulas it is easy to see that in these cases the structure of a $\Hycomm$-algebra gives raise to a full structure of a Frobenius manifold, that is, we also have a scalar product and homogeneity with all the necessary properties.

In~\cite{DotShaVal} the structure of a $\Hycomm$ algebra is obtained in a different way, using a general result of Drummond-Cole and Vallette in~\cite{DruVal} on a homotopy Frobenius structure on the cohomology of a $\BV$-algebra, where the homotopy transfer of the $\BV$ operator $\Delta$ vanishes. In fact this result in~\cite{DruVal} is completely parallel to ours, as it is shown in~\cite{DotShaVal}, though exact match of the formulas will require more work.

\section{Circle action}
\label{sec::circle}
In this section we compare our definition of a homotopy quotient with a categorical one 
and show how this affects to the Chern classes.

\subsection{Homotopy quotient in Topology}
Consider a topological space $X$ with a chosen action of the group $S^1$.
If the action of $S^1$ is not free then the quotient space $X/S^{1}$ is not well defined.
Therefore in order to define the quotient one has to replace the space $X$ by a homotopy equivalent space $X\times ES^1$
with the free action of $S^1$. 
Recall that $ES^1$ is a contractible space with the free $S^1$-action.
The corresponding bundle $ES^{1}\stackrel{S^1}{\rightarrow} BS^1$ is called the universal $S^1$-bundle
and it's base $BS^1$ is called classifying space and known to coincide with ${\mathbb{C}\mathrm{P}^{\infty}}$.
The homotopy quotient ``$X/S^{1}$'' is defined as a factor $\frac{X\times ES^1}{S^1}$. 

There is another categorical definition which we find useful to recall.
Denote by $S^{1}\tt\Top$ (resp.$\Top$) the homotopy categories of topological spaces with (and without) action of $S^1$.
There is a natural exact functor $Triv^{S^1}:\Top \rightarrow S^{1}\tt\Top$ which assigns a trivial action of $S^1$ 
on any topological space.
The left adjoint functor to the functor $Triv^{S^1}$ is called the homotopy quotient by $S^1$:
$$
{\mathrm{Hom}}_{\Top} ( X/ S^1, Y) \simeq  {\mathrm{Hom}}_{S^1\tt\Top}(X,Triv^{S^1}(Y))
$$
In particular, if $X$ is isomorphic to the direct product $Z\times S^{1}$ the homotopy quotient $X/S^1$ is isomorphic to $Z$.

\subsection{Homotopy quotient for algebraic operads}
\label{sec::homotopy_quotients_def}
Replace the category $\Top$ by the category $dg\tt\mathcal{O}p$ of differential graded operads.
The cohomology ring of the circle is the Grassman algebra $\kk[\Delta]$
with one odd generator of degree $-1$  such that $\Delta^2=0$.
The category ${\Delta\tt dg\tt\mathcal{O}p}$ of dg-operads with a chosen embedding of the Grassman algebra $\kk[\Delta]$
replaces the category of $S^{1}\tt\Top$ of topological spaces with a circle action.
An object of the category ${\Delta\tt dg\tt\mathcal{O}p}$ is a dg-operad 
 with a chosen unary operation of degree $-1$, such that its square is equal to zero.
Any dg-operad $\Q$ admits a trivial map $\kk[\Delta]\rightarrow \Q$ with $\Delta\mapsto 0$.
This defines a functor $Triv^{\Delta}: dg\tt\mathcal{O}p \rightarrow \Delta\tt dg\tt \mathcal{O}p$.
\begin{definition}
\label{def::hom_quotient}
 The homotopy quotient by $\Delta$ is the left adjoint functor to the enrichment by trivial embedding of Grassman algebra:
I.~e. it is a functor $(\tt)/\Delta: \Delta\tt dg\tt\mathcal{O}p \rightarrow dg\tt \mathcal{O}p$ such that for any pair of operads
$\P,\Q$ there exists a natural equivalence 
$$
{\mathrm{Hom}}_{dg\tt \mathcal{O}p}(\P/\Delta,\Q)\simeq {\mathrm{Hom}}_{\Delta\tt dg\tt\mathcal{O}p} ( \P, Triv^{\Delta}(\Q)) 
$$
which is functorial in $\P$ and in $\Q$.
\end{definition}
In Section~\ref{sec::bv/delta} we have already chosen a particular model of the homotopy quotient by $\Delta$.
Let us show that this model indeed satisfies the adjunction property required by Definition~\ref{def::hom_quotient}.
First, let us repeat the construction from Section~\ref{sec::bv/delta} in a general setting.

Any given operad $\Q$ with a chosen unary operation $\Delta\in\Q(1)$ (such that $\Delta^2=0$)
may be extended
by a collection of unary operations $\phi_i$, $i=1,2,\dots$, of homological degree $\deg \phi_i = -2i$, 
and the differential prescribed by Equation~\eqref{eq:d-delta}. 
We remind that the generating series in $z$ of 
the sequence of identities on the commutators $[d,\phi_i]$  
defines a differential:
$$
\exp(-\phi_1 z^1 -\phi_2 z^2 -\ldots) d \exp(\phi_1 z^1 +\phi_2 z^2 +\ldots) = d + z\Delta.
$$
Note that the differential decreases the degree in $\phi_i$'s by $1$ and increases the degree in $\Delta$ by $1$. 
We want to keep this property in the notation for the differential; therefore, we denote it by $\Delta\frac{\d}{\d\phi}$
and this notation should be understood just as a single symbol.

\begin{proposition}
\label{lem:adjunction}
 The functor that sends an operad $\Q$ with a chosen squared zero unary operation $\Delta$ to the dg-operad 
$\left(\Q\star\kk\langle\phi_1,\phi_2,\ldots\rangle,\Delta\frac{\d}{\d\phi}\right)$
gives a particular model of the homotopy quotient $\Q/\Delta$.
I.~e. the twisted free product with $\phi$'s is the left adjoint functor to the trivial action of $\kk[\Delta]$.
\end{proposition}
 \begin{proof}
Recall that $\kk[\Delta]$ is a skew commutative algebra with one odd generator $\Delta$,
where the skew-commutativity implies the relation $\Delta^2=0$.
This algebra is Koszul and its Koszul dual is the free algebra $\kk[\delta]$ with one even generator of degree $2$.
The free product of the Grassman algebra $\kk[\Delta]$ and the free algebra $F$ generated by 
the augmentation ideal of Koszul dual coalgebra together with Koszul differential is acyclic. 
We state that $\phi_i$'s  
is just a one possible way to find generators in the free algebra
generated by the augmentation ideal of $\kk[\delta]$ and the differential $\Delta\frac{\d}{\d\phi}$
is the corresponding description of the Koszul differential.

Therefore, the free product $\kk[\Delta]\star\kk\langle\phi_1,\phi_2,\ldots\rangle$ is a factor of the free associative algebra
generated by $\Delta$ and $\phi_i$, $i=1,2,\ldots$ by the unique relation $\Delta^2=0$.
This algebra is acyclic with respect to the differential $\Delta\frac{\d}{\d\phi}$, admits the natural splitting:
\begin{equation}
\label{eq::ES1::noncom}
\kk \stackrel{1\mapsto 1}{\hookrightarrow} 
\left(\kk\langle\Delta,\phi_1,\phi_2,\ldots\rangle/(\Delta^2),\Delta\frac{\d}{\d\phi}\right) \stackrel{\Delta,\phi_i\mapsto 0}{\twoheadrightarrow} \kk 
\end{equation}
and satisfies the following universal categorical property:
For any dg-algebra $(A,d_A)$ with a chosen dg-subalgebra $(\kk[\Delta_A],0)$ 
there exists a map of dg-algebras $\varphi_{A}:(\kk[\Delta]\star\kk\langle\phi_1,\phi_2,\ldots\rangle,\Delta\frac{\d}{\d\phi})\rightarrow (A,d_A)$
that sends $\Delta\mapsto \Delta_A$ and is functorial with respect to $A$.
One should think about the dg-algebra $(\kk[\Delta]\star\kk\langle\phi_1,\phi_2,\ldots\rangle,\Delta\frac{\d}{\d\phi})$
as a noncommutative algebraic replacement of the universal bundle $ES^1$.

We will come back to the connection with the universal bundle in the next Section~\ref{sec::chern}.
\end{proof}

For any given dg-operad $(\P,d_{\P})$ we define the quasi-isomorphic inclusion of dg-operads
\begin{equation}
\label{eq::eps_P}
\varepsilon_{\P}:(\P,d_{\P}) \rightarrow \left(\P\star\kk[\Delta]\star\kk\langle\phi_1,\phi_2,\ldots\rangle, d_{\P} + \Delta\frac{\d}{\d\phi}\right)
\end{equation}
that sends $\P$ to $\P$;
and with any dg-operad $(\Q,d_{\Q})$ with a chosen unary operation $\Delta_{\Q}\in\Q(1)$
we associate the projection of $S^{1}\tt$dg-operads
\begin{equation}
\label{eq::eta_Q}
\eta_{\Q}: \left(\Q\star\kk\langle\phi_1,\phi_2,\ldots\rangle\star\kk[\Delta], d_{\Q}+(\Delta-\Delta_{\Q})\frac{\d}{\d\phi} \right) 
\rightarrow (\Q,d_{\Q}) 
\end{equation}
that sends identically $\Q$ to $\Q$, $\Delta\mapsto\Delta_{Q}$ and $\phi_i$ maps to $0$ for all $i$.
\begin{lemma}
The morphisms $\varepsilon_{\P}$ and $\eta_{\Q}$ are quasi-isomorphisms for all $\P$ and $\Q$.
\end{lemma}
\begin{proof}
 The proof follows from the acyclicity of the dg-algebra 
$(\kk[\Delta]\star\kk\langle\phi_1,\phi_2,\ldots\rangle,\Delta\frac{\d}{\d\phi})$.
\end{proof}

Let us also give one more explanation on why we call the data $\phi_i$'s by a choice of trivialization of the action of $S^1$.
The action of $S^1$ on a topological space $X$ is encoded in the fibration $X\times ES^1\stackrel{S^1}{\rightarrow} B$.
The trivialization of the $S^1$ action is the isomorphism of this fibration and the trivial one.
I.e. is given via isomorphism $\Phi$ of the base $B$ and the product $X\times BS^1$.
The algebraic counterpart of this isomorphism looks as follows:
$$
\Phi: Tor_{\ldot}^{\kk[\Delta]}(V^{\udot},\kk)\stackrel{\cong}{\longrightarrow} V^{\udot}\otimes Tor_{\ldot}^{\kk[\Delta]}(\kk,\kk)
$$
where $V^{\udot}=C^{\udot}(X)$. 
The trivial module $\kk$ admits a Koszul resolution 
$$
(\kk[\Delta]\otimes \kk[z], z\frac{\d}{\d\Delta} )\rightarrow \kk
$$
and we ends up with the following isomorphism of complexes:
$$
\Phi=\Phi(z): (V^{\udot}[z],d +z\Delta) \rightarrow (V^{\udot}[z],d)
$$
that is called the trivialization of the action of $\Delta$ (the trivialization of $S^1$-action).

\subsection{Chern character}
\label{sec::chern}
Suppose that $\Q$ is a topological operad with a chosen embedding $S^1\hookrightarrow \Q(1)$.
Note that the latter embedding gives, in particular, the action of ($n+1$) copies of $S^1$ on the space
of $n$-ary operations $\Q(n)$ 
via the substitution on inputs/output of operations.
It is possible to take a homotopy quotient with respect to the action on each particular input/output on the 
space of $n$-ary operations of $\Q$. We denote by $\Q/(\circ_i\Delta)$ the quotient with respect to the action
of $S^1$ on the $i$-th slot. 
Moreover, the $S^1$-action on each particular input produces a canonical $S^1$-fibration 
on the space of $n$-ary operations of the entire quotient $\Q/\Delta$.  
It is simpler to describe the algebraic counterpart of this fibration 
in order to define the first Chern class of this fibration which gives a canonical operation on a factor.
This description will be used later on to give another algebraic description of the $\psi$-classes in the moduli spaces of curves.

We hope that the reader will not be confused about no difference in the notations of the topological operad and 
the corresponding algebraic operad of its singular chains.
From now on $\Q$ means an algebraic operad with a chosen unary odd operation $\Delta$ with $\Delta^2=0$. 
Let $(\Q/\Delta)_{\epsilon_i}$ be the subset of  $n$-ary operations in $\Q/\Delta$ where we take the 
augmentation map 
$$
\epsilon: \kk\langle\phi_1,\phi_2,\ldots\rangle \twoheadrightarrow \kk 
$$
with respect to the $i$-th input of operations. I.~e. we consider only those elements of $\Q/\Delta$ 
which do not contain any nonconstant element from the algebra $\kk\langle\phi_1,\phi_2,\ldots\rangle$.
The natural inclusion of complexes $(\Q/\Delta)_{\epsilon_i}(n) \rightarrow (\Q/\Delta)(n)$  gives the algebraic model
of the  $S^1$-fibration described above.

Let $\frac{\d}{\d \phi_1}$ be the derivation of the algebra $\kk\langle\phi_1,\phi_2,\ldots\rangle$ that sends 
the generator $\phi_1$ to $1$
and all other generators $\phi_i$ for $i\geq 2$  to zero.
Let $\circ_i\frac{\d}{\d \phi_1}$ be the derivation of the set of $n$-ary operations $\Q/\Delta(n)$ obtained by applying the 
derivation $\frac{\d}{\d \phi_1}$ in the $i$-th slot of the operation.
\begin{proposition}
\label{prop::psi_algebraic}
The derivation $\circ_i\frac{\d}{\d \phi_1}$ of the complex of $n$-ary operations $\Q/\Delta(n)$
represents the evaluation of the first Chern class of the $S^1$-fibration over $\Q/\Delta$ 
associated to the $S^1$ action in the $i$-th slot.
\end{proposition}
\begin{proof}
The Chern class is defined as a generator of the cohomology of the Eilenberg-Maclein space $BS^1$ 
(the base of the universal bundle). In order to switch to algebra we have to reformulate the required 
categorical properties of the universal bundle in algebraic terms.
First, let us formulate the desired property in the category of commutative dg-algebras since 
the homology functor is the map from topological spaces to commutative algebras.
The commutative dg-algebra $(\kk[\Delta,u],\Delta\frac{\d}{\d u})$
is an acyclic dg-algebra that satisfy the universal property:
for any commutative dg-algebra $(A,d_A)$ with a chosen dg-subalgebra $(\kk[\Delta],0)$ 
there exists a map of dg-algebras $\varphi_{A}:(\kk[\Delta,u],\Delta\frac{\d}{\d u})\rightarrow (A,d_A)$
that sends $\Delta\mapsto \Delta$ and is functorial with respect to $A$.
The generator $u$ is the multiplicative generator of $H^{\udot}(BS^1;\kk)$ and the derivation 
$\frac{\d}{\d u}$ coincides with the evaluation of the first Chern class of the circle bundle.
Second, we notice that 
the dg-algebra  $(\kk[\Delta]\star\kk\langle\phi_1,\phi_2,\ldots\rangle,\Delta\frac{\d}{\d\phi})$
is an acyclic dg-algebra satisfying the same universal property, but in the category of noncommutative algebras.
Now the generator $\phi_i$ corresponds to the additive generators of $H^{2i}(BS^1;\kk)$.
There exists a natural quasi-iso projection between these two algebras:
\begin{equation}
\label{eq::phi_i->u}
\xymatrix{
ab: \left(\kk\langle\Delta,\phi_1,\phi_2,\ldots\rangle / (\Delta^2=0), \Delta\frac{\d}{\d \phi}\right)
\ar@{->>}[r] 
&
\left( \kk[\Delta,u], \Delta\frac{\d}{\d u}\right).
}
\end{equation}
that sends $\Delta$ to $\Delta$, $\phi_1$ to $u$ and all other $\phi_i$, for $i\geq 2$ to $0$.
Moreover, the derivation $\frac{\d}{\d\phi_1}$ of the left hand side of~\eqref{eq::phi_i->u}
 commutes with the differential and 
is mapped to the derivation $\frac{\d}{\d u}$ on the right and, therefore, coincides (on the homology level)
with the evaluation of the first Chern class.
\end{proof}

\section{Operads involved: definition and notation}
\label{sec::operad::definitions}

In this section we recall the definitions of algebraic operads  that 
correspond to the topological operads of open and closed moduli spaces of curves of zero genus.
We follow the papers of Getzler~\cite{Getzler_genus0,Getzler_grav}.
Since we want to work with precise formulas, we specify 
algebraic generators and relations in these operads.

We also give definitions of the operads involved in Section~\ref{sec:diagrammatic} 
in the main commutative diagram~\eqref{eq::diag::Hycom_BV}
used to derive the equivalence of $\Hycomm$ and $\BV/\Delta$.

We use the notation $\circ_l$ for the operadic compositions in the $l$'th slot.
I.~e. for an operad $\P$ and a pair of finite sets $I,J$ 
the composition $\circ_{l}:\P(I\sqcup\{l\})\otimes\P(J)\rightarrow \P(I\sqcup J)$ is a substitution of operations
from $\P(J)$ into the slot $l$ of operations from $\P(I\sqcup\{l\})$.
The corresponding cocomposition map $\P^{\dual}(I\sqcup J)\rightarrow \P^{\dual}(I\sqcup\{l\})\otimes\P^{\dual}(J)$
for the dual cooperad $\P^{\dual}$ will be denoted by $\mu_l$ or just by $\mu$ if the precise index becomes clear from the context.

There are two standard ways to think of elements of an operad/cooperad in terms of its (co)generators.
The first way in terms of tree monomials  represented by planar trees 
and the second one is in terms of compositions/cocompositions of operations presented by formulas with brackets.
Our approach is somewhere in the middle: in most cases, we prefer (and strongly encourage the reader)
to think of tree monomials, but to write formulas required for definitions and 
proofs in the language of operations since it makes things more compact.
While using the language of operations/cooperations we always suppose that the (co)operation that is attached to the root vertex 
is written in the leftmost term.

\subsection{$\BV$ and framed little discs operad}
\label{sec::BV::framed_little_discs}

The space of configurations of the small little discs without intersections inside the unit disc form
one of the most well known topological operad.
The boundary of the unit disc is considered as an output and the boundaries of the inner small discs 
are considered as inputs. This means that the composition rules are defined by gluing the boundary of the inner disc of 
the outgoing operation with the outer boundary of the incoming operation.
Following May~\cite{May} we use the name $E_d$ for this operad where $d$ is a dimension of the disc.
We restrict ourself to the case $d=2$.
It is also known that operad $E_2$ is formal over $\mathbb{Q}$ (see e.g.~\cite{Tamarkin_form,Kontsevich_form}) 
and its homology operad
coincides with the operad of Gerstenhaber algebras.

Recall, that the operad $\Gerst$ of Gerstenhaber algebras is a quadratic operad generated by two binary operations:
the commutative associative multiplication and the Lie bracket of degree $-1$.
The quadratic operadic relations consists of:
 the associativity of multiplication, Jacobi identity for the bracket
and the Leibniz identity for their composition:
$$
[a\cdot b, c] = \pm[a,c]\cdot b \pm a\cdot [b,c]
$$
Moreover, the space of $n$-ary operations $\Gerst(n)$ form a coalgebra, such that the composition maps 
are compatible with comultiplications in these coalgebras. We will come back later to this description
of the Gerstenhaber operad in Section~\ref{sec::grav}.

Let us mark a point on the boundary circle of each inner disk in a configuration from $E_2(n)$.
This leads to a description of the space of $n$-ary operations of the operad of framed little discs
which we denote by $FE_2$. The composition rules in $FE_2$ are also defined 
by gluing the boundary of the inner disc of 
the outgoing operation with the outer boundary of the incoming operation but now the marked point of the inner circle should
be glued with the north pole of the outer circle. I.~e. one has to rotate the incoming configuration with respect to the angle 
prescribed by the marked point in the inner circle of the outgoing configuration.
This operad is also known to be formal (\cite{Severa,Salvatore}) and the homology operad coincides with 
the operad of Batalin-Vilkovisky algebras (shortly denoted by $\BV$).

The operad $\BV$ is generated by the binary commutative associative multiplication
and a unary operation $\Delta$ of degree $-1$ such that $\Delta^2=0$ and $\Delta$ is a differential 
operator of the second order with respect to the multiplication.
The latter statement is equivalent to the following so-called $7$-term relation:
$$
\Delta(abc) - (\Delta(ab)c + \Delta(bc)a +\Delta(ca)b) + (\Delta(a)bc + \Delta(b)ca +\Delta(c)ab) = 0
$$
We omit the precise signs that comes from the Koszul sign rule in the $\Z$-graded settings. 

Note that the topological description of the operad of framed little discs is presented as a semi-direct 
product (or semi-direct composition) of the little discs operad $E_2$ and the group of rotations $S^1$.
The topological definition of the semi-direct composition of a group and an operad is given in~\cite{Salvatore_fE_d}.
In our case, the group is $S^1$, and the algebraic counterpart consists of the semi-direct product 
of the Gerstenhaber operad with a free skew-commutative algebra $\kk[\Delta]$ generated by a unique generator $\Delta$ of degree $-1$.
This leads to the following equality of operads:
$$
\BV = \Gerst \ltimes \kk[\Delta].
$$
Here the semi-direct product $\Gerst \ltimes \kk[\Delta]$ means the operad generated by
the binary commutative multiplication, Lie bracket and  unary operation $\Delta$ subject to relations for 
multiplication and bracket as in $\Gerst$, $\Delta^2=0$ as in the skew-commutative algebra $\kk[\Delta]$
and the following commutation relation between $\Delta$ and generators of $\Gerst$:
$$
\{ \Delta, \text{ multiplication}\} = \text{ Lie bracket} , \qquad 
\{ \Delta, \text{ Lie bracket}\} = 0
$$
The patterned brackets denotes the operadic commutator. In particular, the operadic commutator of a unary operation
$\Delta$ and an $n$-ary operation $\alpha(\tt,\ldots,\tt)$ means the following expression with $n+1$ terms:
$$
\{\Delta, \alpha(\tt,\ldots,\tt)\}:=
\Delta\circ\alpha(\tt,\ldots,\tt) - \sum_{i=1}^{n} \alpha(\tt,\ldots,\tt)\circ_i \Delta.
$$
We will come back later to the precise description of the spaces of $n$-ary operations of $\Gerst(n)$ and $\BV(n)$ 
in Sections~\ref{sec::grav} and \ref{sec::BV=Gerst*Delta} respectively.
 
\subsection{Closed moduli spaces of zero genus}
\label{sec::open_close}
The union of spaces of compactified moduli spaces of curves of zero genus form an operad.
This operad is formal. Its homology is called $\Hycomm$ (the operad of hypercommutative algebras).

The algebraic description of the operad $\Hycomm$ looks as follows.
The operad $\Hycomm$ has one generator in each arity grater or equal to 2.
The generator $\mm_k$ of arity $k$ is of degree $(4-2k)$ and is given by the fundamental cycle $\mm_k:=[\oM_{0,k+1}]$.
The generators satisfy the following quadratic relations (here $a,b,c,x_1,\dots,x_n$, $n\ge 0$, are elements of a $\Hycomm$-algebra):
\begin{equation}\label{hycom_rel}
\sum_{S_1\amalg S_2=\{1,\dots,n\}} \pm \mm_{|S_2|+2}(\mm_{|S_1|+2}(a,b,x_{S_1}),c,x_{S_2})
= \sum_{S_1\amalg S_2=\{1,\dots,n\}} \pm \mm_{|S_2|+2}(a,\mm_{|S_1|+2}(b,c,x_{S_1}),x_{S_2}).
\end{equation}
Here, for a finite set $S=\{s_1,\dots,s_k\}$, $x_S$ denotes for $x_{s_1},\dots,x_{s_k}$, and $\pm$ means the Koszul
sign rule.
Let us define a family of binary operations $m_x(\tt,\tt)$ on $V$ parametrized by the same space $V$:
$$
\forall x\in V \text{ let }
m_{x}(a,b):= \sum_{n\geq 0} \frac{1}{n!}\mm_{n+2}(a,b,x,\ldots,x)
$$
Then Equation~\eqref{hycom_rel} is equivalent to the associativity of the multiplication $m_x(\tt,\tt)$ for all $x\in V$. 
This observation explains the relation between hypercommutative algebras and Frobenius manifolds.

The first Chern class of the tangent bundle at the $i$'th marked point on $\oM_{0,n+1}$ is usually denoted by $\psi_i$.
Let $\mm_{n}^{d_0 d_1\ldots d_n}$ be the cycle corresponding to the evaluation 
of the product of $\psi$-classes of corresponding degrees on the fundamental cycle of the space of curves:
$$\mm_{n}^{d_0 d_1\ldots d_n} := \psi_0^{d_0}\psi_1^{d_1}\ldots\psi_n^{d_n}[\oM_{0,n+1}].$$
These classes satisfy the so-called \emph{Topological Recursion Relations} that are quadratic linear
relations in the operadic sense:

\begin{align*}
& \mathrm{m}^{(d_0+1)d_1\cdots d_n} + \mathrm{m}^{d_0\cdots d_{i-1}(d_i+1) d_{i+1}\cdots d_n} = 
\sum_{\begin{smallmatrix}
S_1\sqcup S_2 
\sqcup \{0,i\} \\
= \{0,\dots,n\} 
\end{smallmatrix}} 
\mathrm{m}^{d_0d_{S_1}0} \circ_{|S_1|+1} \mathrm{m}^{0d_id_{S_2}} 
& & \forall 1\leq i \leq n ; \\
& \mathrm{m}^{(d_0+1)d_1\cdots d_n} = 
\sum_{\begin{smallmatrix}
S_1\sqcup S_2 
\sqcup \{0,i,j\} \\
= \{0,\dots,n\} 
\end{smallmatrix}} 
\mathrm{m}^{d_0d_{S_1}0} \circ_{|S_1|+1} \mathrm{m}^{0d_id_jd_{S_2}} 
& & \forall 1\leq i, j \leq n. 
\end{align*}
Here we denote by $d_S$, $S=\{s_1,\dots,s_k\}$, the sequence $d_{s_1}\cdots d_{s_k}$. We will come
back later to TRR equations in Section~\ref{sec::TRR}. For more details see~\cite{Man}.

\subsection{Open moduli spaces of zero genus}
\label{sec::open:moduli}
The shifted homology of the union of spaces of open moduli spaces of curves of zero genus also form a formal operad.
The corresponding algebraic operad is called $\Grav$
(the operad of gravity algebras). It was studied by Getzler in~\cite{Getzler_genus0}, in particular, he proved that $\Grav$ and $\Hycomm$ are Koszul dual to each other.

An algebra over~$\Grav$ is a chain complex with graded anti-symmetric products
\begin{equation} \label{eq:mmg-first}
\mmg_n[x_1,\dots,x_n]\colon A^{\otimes n}\to A
\end{equation} 
of degree $2-n$ that satisfy the relations: 
\begin{align}
\label{relgrav}
& \sum_{1\le i<j\le k}
\pm \mmg_{k+l-1}[\mmg_2[a_i,a_j],a_1,\dots,\widehat{a_i},\dots,\widehat{a_j},\dots,a_k,
b_1,\dots,b_l] 
\\ \notag &
= \begin{cases} \mmg_{l+1}[\mmg_k[a_1,\dots,a_k],b_1,\dots,b_l] , & l>0 , \\
0 , & l=0, 
\end{cases} 
\end{align}
for all $k>2$, $l\ge0$, and $a_1,\dots,a_k,b_1,\dots,b_l\in A$. For example, in the case of $k=3$ and $l=0$, 
we obtain the Jacobi relation for~$\mmg_2[\cdot,\cdot]$.

Once again, Getzler proved in~\cite{Getzler_genus0} that $\Hycomm$ and $\Grav$ are Koszul dual operads.
Moreover for all $n\geq 2$ the generators $\mm_n\in\Hycomm(n)$ 
and $\mmg_n\in\Grav(n)$ are Koszul dual generators in these operads.\footnote{
We are a bit cheating here because Koszul duality gives the duality between generators and cogenerators.
But there is no reason to separate generators and cogenerators in our particular situation
because the corresponding 
subspaces of homological degrees $2n-4$ and $n-2$ in $\Hycomm(n)$ and $\Grav(n)$ respectively are one-dimensional.}
In particular, the associativity relation for the commutative multiplication $\mm_2\in\Hycomm(2)$ is a relation 
Koszul dual to the Jacobi relation for the Lie bracket $\mmg_2\in\Grav(2)$.

Let us also mention another result due to Getzler which hints the desired connection between $\Hycomm$ and $\BV$.
The space of cotangent lines at the $i$-th marked point of curves from $\mathcal{M}_{0,n+1}$
forms a line bundle over the open moduli space $\mathcal{M}_{0,n+1}$.
Consider the product of corresponding $(n+1)$ principal $U(1)$-bundles over $\mathcal{M}_{0,n+1}$,
where the factors are numbered by the marked points.
\begin{statement}(\cite{Getzler_grav})
\label{stat::S1->BV->Grav}
The homology of the total space of the $(S^1)^{\times (n+1)}$-bundle over $\mathcal{M}_{0,n+1}$ (associated 
with the product of cotangent lines at the marked points of a curve) coincides with the
space of $n$-ary operations in the operad $\BV$. 
\end{statement}
We give the algebraic counterpart of this statement in the next section.

\subsection{Gerstenhaber and gravity operads}
\label{sec::grav}

Getzler observed that the $S^1$-equivariant homology of the Gerstenhaber operad is isomorphic to the gravity operad.
This statement has very clear geometric background, see~\cite{Getzler_grav}, since the Gerstenhaber operad is the homology of the little disk operad. 
We recall the algebraic counterpart of this isomorphism.

It is easier to compute the cohomology rather than homology of the space of little disks 
(it was done by Arnold in~\cite{Arnold}). This way we obtain a description of the cooperad dual to the Gerstenhaber operad. 
The space of $n$-ary cooperations of the cooperad $\Gerst^{\dual}$
form a so-called Orlik-Solomon algebra:
\begin{equation*}
\Gerst^{\dual}(n):=\frac{\kk\left[\left\{w_{ij}\right\}_{1\leq i,j\leq n,\ i\ne j}\right]}{
\left( w_{ij} - w_{ji},
w_{ij} w_{jk} + w_{jk} w_{ki} + w_{ki} w_{ij}
\right)
}
\end{equation*}
Here we mean that $\Gerst^{\dual}(n)$ is a quotient modulo an ideal of the free graded commutative algebra generated by $w_{ij}$, $\deg\, w_{ij}=1$. 

The cooperad structure satisfies the Leibniz rule with respect to the product structure in the algebra $\Gerst^{\dual}(n)$, $n\geq 2$.
Therefore, it is enough to define the cooperad structure 
$\mu:\Gerst^{\dual}(I\sqcup J) \rightarrow \Gerst^{\dual}(I\sqcup\{*\})\otimes\Gerst^{\dual}(J)$ on the generators $w_{ij}$. By definition,
\begin{equation}
\mu(w_{ij}) = \begin{cases}
               w_{ij}\otimes 1 \text{, if } i,j\in I; \\
               w_{i*}\otimes 1 \text{, if } i\in I, j\in J; \\
               1\otimes w_{ij}\text{, if } i,j\in J.
              \end{cases}
\label{eq::coordinates_in_BV}
\end{equation}

There is an action of the circle $S^1$ on the little discs operad via the rotation of the outer circle.
The corresponding coaction of the generator $\Delta$ of the first cohomology of the circle $S^1$ on the space $\Gerst^{\dual}(n)$ 
is given by the following operator:
\begin{equation}
\label{eq::Delta_act_gerst}
 \frac{\d}{\d{w}} := \sum_{1\leq i<j \leq n} \frac{\d}{\d w_{ij}}.
\end{equation}
The action of the operator $\frac{\d}{\d{w}}$ on $\Gerst^{\dual}$ is dual to the action of the operator $\Delta$ on $\Gerst$.
\begin{statement}(\cite{Getzler_grav})
The action of the operator $\Delta$ is free on the Gerstenhaber operad $\Gerst$. The image of $\Delta$ coincides with its kernel and is isomorphic to  the gravity operad.
\label{stat::Gerst_Grav}
\end{statement}

Let us define a homotopy model for the gravity operad. We use standard manipulations with equivariant homology. 
We consider the free polynomial algebra $\kk[\delta]$, $\delta$ is even, as the Koszul dual of the algebra $\kk[\Delta]$.

\begin{definition} By $\kk[\delta]\otimes\Gerst$ we denote a dg-operad with 
$(\kk[\delta]\otimes\Gerst)(n):= \kk[\delta]\otimes(\Gerst(n))$ for $n\geq 2$ and
$(\kk[\delta]\otimes\Gerst)(1):= \Gerst(1) = \kk$.
The composition is defined by
\begin{equation}
(\delta^{a}\otimes \alpha )\circ (\delta^b\otimes \beta ) := \delta^{a+b} \alpha\circ\beta \qquad
 \text{ for } \alpha,\beta \in \Gerst(n).
\end{equation}
The $\BV$-operator defines the differential $\delta\Delta\colon \delta^{a}\alpha\mapsto \delta^{a+1}\Delta(\alpha)$ on this operad.
\end{definition}

Let us rephrase Statement~\ref{stat::Gerst_Grav} in the language of cooperads. 
We use the notation $u$ for the even variable linear dual to $\delta$
and $\kk[u]\otimes\Gerst^{\dual}(n)$ for the space linear dual to $\kk[\delta]\otimes\Gerst(n)$ for all $n\geq 2$.
\begin{lemma}\label{lem::gerst->grav}
 The augmentation map of dg cooperads 
\begin{equation}
\varepsilon:(\kk[u]\otimes\Gerst^{\dual},\frac{\d}{\d{u}}\frac{\d}{\d{w}})
 \twoheadrightarrow  
\left( \Grav^{\dual} , 0\right) 
\label{eq::augmentation:Gerst->Grav}
\end{equation}
that maps $u\mapsto 0$ and 
$\Gerst^{\dual}\twoheadrightarrow 
\Gerst^{\dual}/(Im\frac{\d}{\d w}) = \Grav^{\dual}$
is a quasi-isomorphism.
\end{lemma}

In particular, the map $\varepsilon$ maps any basic element $w_{ij}\in\Gerst^{\dual}(n)$
 to the unique $n$-ary cogenerator $\mmg_n$ of the gravity cooperad.
The precise homological grading is discussed in the next section.

\subsection{Bar complexes}
\label{sec::Bar_Gerst}
\label{sec::bar_complex}

In this section we recall the general definition of cobar complex and the precise formulation of Koszul self-duality 
for the operad $\Gerst$ and Koszul resolution of $\Hycomm$ via a cobar complex of $\Grav$.

Consider a cooperad $\P^{\dual}$ with a cocomposition $\mu: \P^{\dual} \rightarrow \P^{\dual}\circ \P^{\dual}$.
Let $\P^{\dual}_{+}$ be the augmentation ideal. In all our examples $\P^{\dual}(1)=\kk$ and 
the augmentation ideal $\P^{\dual}_{+}$
is equal to $\oplus_{n\geq 2}\P^{\dual}(n)$.
The cobar complex $\Bar(\P^{\dual})$ is a free dg-operad generated by the shifted space $\P^{\dual}_{+}[-1]$.
The cocomposition $\mu$ defines a differential of degree $1$ on generators. Using the Leibniz rule we extend it 
to the whole cobar complex $\Bar(\P^{\dual})$.

In~\cite{Getzler_Jones} it is proved that the operad $\Gerst$ is Koszul self-dual 
up to an appropriate even shift of homological degree. Pure algebraic proof of that fact was first given in~\cite{Markl_distr_law}.
Let us specify the desired homological shift. Note that Getzler defined two different types of grading on $\Grav$ in \cite{Getzler_genus0,Getzler_grav}. They differ by the even shift $s^2$ on the Gerstenhaber operad that we define now.  
By $s^{2}\Gerst^{\dual}$ we denote a quadratic cooperad whose $n$-th space is given by $s^2\Gerst^{\dual}(n) = \Gerst^{\dual}(n)[2n-2]$.
In other words, we can define $s^2\Gerst^{\dual}$ as a quotient of a free cooperad generated by binary operations modulo an ideal exactly in the same way as $\Gerst^{\dual}$, but we shift by $2$ the homological degree of the binary generators. 

The Koszul self-duality means that the natural projection of dg-operads
\begin{equation}
\label{eq::BGerst->Gerst}
\pi: \left(\Bar(s^{2}\Gerst^{\dual}),\mu\right) \twoheadrightarrow \left(\Gerst,0\right)
\end{equation}
 is a quasi-isomorphism.
Here the map $\pi$ interchanges the multiplication and the bracket. In particular, under $\pi$
\begin{align}
\label{eq::pi::BGerst->Gerst}
w_{12} \in \Gerst^{\dual}(2) & \mapsto  \text{multiplication} \\ \notag
1\in \Gerst^{\dual}(2) & \mapsto  \text{Lie bracket}  \\ \notag
\Gerst^{\dual}(k) & \to  0 \text{ for } k >2.
\end{align}

In order to give a similar construction for the resolution of the operad $\Hycomm$, we consider the cobar complex
of the equivarient model of the gravity operad:
\begin{equation}
 \Bar(\kk[u]\otimes s^2\Gerst^{\dual}) \stackrel{\varepsilon}{\twoheadrightarrow} \Bar(\Grav) 
\stackrel{\kappa}{\twoheadrightarrow} \Hycomm. 
\label{eq::res_gerst_to_hycom}
\end{equation}
The differential $d$ on $\Bar(\kk[u]\otimes s^2\Gerst^{\dual})$ is a sum of two parts.
The first summand is equal to the inner differential $\frac{\d}{\d{u}}\frac{\d}{\d{w}}$.
The second summand is given by cocomposition $\mu$ defined by Equation~\eqref{eq::coordinates_in_BV}.
For example, on a generator $\frac{u^k}{k!} f(w_{ij})$, where $f$ is a monomial in $w_{ij}$, $1\leq i\neq j \leq n$,  
the differential is given by 
\begin{equation*}
 d\left( \frac{u^k}{k!} f(w_{ij}) \right) = \frac{u^{k-1}}{(k-1)!} \sum_{i,j} \frac{\d f}{\d w_{ij}} +
\sum_{\begin{smallmatrix}I\sqcup J = [n], |J|\geq 2, |I|\geq 1,\\k_1+k_2=k\end{smallmatrix}} (-1)^{\deg_{w}{f^{I}}} \frac{u^{k_1}}{k_1!}f^{I}\otimes \frac{u^{k_2}}{k_2!}f^{J}
\end{equation*}
Since $f\in \Gerst^\dual(n)$ is a monomial in $w_{ij}$,  
for each decomposition $I\sqcup J = [n]$ we have a uniquely defined pair of monomials
$f^{I}\in \Gerst^\dual(|I|+1)$ and $f^{J}\in \Gerst^\dual(|J|)$.
It is important for Koszul sign rule in future computations to recall once again the degree of a particular generator of the cobar complex:
$$
\deg (\frac{u^k}{k!} f(w_{ij})) = 2-2n +2k +\deg_{w}f + 1= 3-2(n-k) +\deg_{w} f
$$

\subsection{Applying the homotopy quotient and the free product to gravity operad}

Let us apply the composition of functors we defined in Section~\ref{sec::homotopy_quotients_def} to the free dg-model of the
operad of hypercommutative algebras discussed in Equation~\eqref{eq::res_gerst_to_hycom}.
I.~e. in this section we describe the dg-operad which is the homotopy quotient by $\Delta$ of the free product with $\kk[\Delta]$
of the dg-operad $\Bar(\kk[u]\otimes s^2\Gerst^{\dual})$.

Consider first the image of the free product functor $\Bar(\kk[u]\otimes s^2\Gerst^{\dual})\star \kk[\Delta]$.
Note that $\kk[u]$ comes from the cohomology ring of $BS^1$ and, therefore, 
it is natural to define the differential which interacts the action of $\Delta$ and $u$:
\begin{equation*}
\Delta^{ad}\frac{\d}{\d u}\colon \gamma \mapsto \left\{\Delta,\frac{\d \gamma}{\d u}\right\} = \Delta \circ \frac{\d \gamma}{\d u} 
 -  \sum_{i=1}^{n} \pm \frac{\d \gamma }{\d u}\circ_{i} \Delta.
\end{equation*}
That is, the operator $\frac{\d}{\d u}$ acts on $n$-ary operation $\gamma$  
and $\Delta^{ad}\frac{\d}{\d u}$ acts as the commutator of 
$\frac{\d \gamma}{\d u}$ and $\Delta$. Note that operators $\Delta^{ad}$ and $\frac{\d}{\d u}$ commute.

The following corollary follows directly from the proof of Proposition~\ref{lem:adjunction}:
\begin{corollary}
\label{lem::delta/delta}
The natural projection that takes $\Delta,\phi_1,\phi_2,\ldots$ to $0$ is a quasi-isomorphism of dg-operads 
\begin{equation}
 \label{eq::j_def}
\left(
\frac{
\Bar(\kk[u]\otimes s^2\Gerst^{\dual})\star \kk[\Delta]}{\Delta},
\frac{\d}{\d{u}} \frac{\d}{\d{w}} + \mu + \Delta^{ad}\frac{\d}{\d u} + \Delta\frac{\d}{\d\phi}
\right)
\longrightarrow
\left(\Bar(\kk[u]\otimes s^2\Gerst^{\dual}), 
\frac{\d}{\d{u}} \frac{\d}{\d{w}} + \mu \right).
\end{equation}
\end{corollary}

The operad $\Bar(\kk[u]\otimes s^2\Gerst^{\dual})\star \kk[\Delta]$ 
will be referred to as an \emph{equivariant cobar complex}.
This operad is spanned by trees whose vertices are marked by elements of 
the cooperad $\kk[u]\otimes s^2\Gerst^{\dual}$ and some edges are marked by $\Delta$.

\subsection{BV and semi-direct composition of operads}
\label{sec::BV=Gerst*Delta}
In this section we recall the presentation of the $\BV$ operad in terms of the semi-direct composition.
The topological definition of the semi-direct composition of a group and an operad is given in~\cite{Salvatore_fE_d}.
In our case, the group is $S^1$, and the algebraic counterpart consists of the semi-direct composition 
of the Gerstenhaber operad with a free algebra $\kk[\Delta]$ generated by a unique generator $\Delta$ of degree $-1$.
As we have already mentioned in Section~\ref{sec::grav}, the circle acts by inner rotations 
of the disc and the corresponding coaction is given by the operator $\frac{\d}{\d w}$ defined by Equation~\rf{eq::Delta_act_gerst}.

We have already mentioned in Section~\ref{sec::BV::framed_little_discs}
 that the operad $\Gerst\ltimes\kk[\Delta]$ coincides with $\BV$. 
Let us specify a bit the description of $\Gerst\ltimes\kk[\Delta]$.
The space of $n$-ary operations of $\Gerst\ltimes\kk[\Delta](n)$
 is equal to $\Gerst(n)\otimes\kk[\Delta_1,\ldots,\Delta_n]$. In particular, $\Gerst\ltimes\kk[\Delta](1)=\kk[\Delta]$.
By definition, for any $\gamma\in\Gerst(n)$ we have:
\begin{align*}
\gamma\circ_i \Delta & := \gamma\otimes \Delta_i; \\ \notag
\Delta\circ \gamma & := \sum_{i=1}^{n} \gamma\circ_i\Delta + \Delta(\gamma),
\end{align*}
where in the last summand we use the action of $\Delta$ on $\Gerst$.
These two formulas allow to extend unambiguously the operadic product on $\Gerst$ to an operadic product on $\Gerst\ltimes\kk[\Delta](n)$.
Moreover, the projection 
$\pi:\Bar(s^2\Gerst^{\dual})\twoheadrightarrow \Gerst$ from Equation~\eqref{eq::pi::BGerst->Gerst} is extended to  
a quasi-isomorphism of semi-direct compositions:
\begin{equation}
\label{eq::def_BGerst->BV}
\pi:
(\Bar(s^2\Gerst^{\dual})\ltimes \kk[\Delta], \mu)   \twoheadrightarrow (\Gerst\ltimes \kk[\Delta],0) = (\BV,0).
\end{equation}

\begin{lemma}
\label{lem::uDelta->semiDelta}
 The natural projection
\begin{equation*}
\epsilon\colon \left( \Bar (\kk[u]\otimes s^2\Gerst^\dual)\star\kk[\Delta],
\frac{\d}{\d{u}} \frac{\d}{\d{ w}} +  \mu + \Delta^{ad}\frac{\d}{\d u}\right)
 \twoheadrightarrow \left(\Bar(s^2\Gerst^\dual)\ltimes \kk[\Delta],  \mu\right)
\end{equation*}
that sends $u\mapsto 0$, $\Delta\mapsto \Delta$, and $\Gerst\to \Gerst$, is a quasi-isomorphism of dg-operads.
\end{lemma}

\begin{proof}
First, we check that $\epsilon$ is a morphism of dg-operad. 
Indeed, a direct computation follows that $\epsilon$ is compatible with the differentials.
Since the cobar complexes are free operads, we immediately get the compatibility with the operadic structures.

Then we consider a filtration by the number of internal edges in cobar complexes both in the source and in the target of $\epsilon$.
The associated graded differential in the target is equal to $0$, and the associated graded differential in the source dg-operad is equal to 
$\frac{\d}{\d{u}} \frac{\d}{\d{ w}} + \Delta^{ad}\frac{\d}{\d u}$.

At that point it is possible to choose a filtration (or rather a sequence of filtrations) in the source dg-operad such that associated graded differential will simplifies further and is equal to is $\Delta^{out}\frac{\d}{\d u}$. Here $\Delta^{out}$ is an operator defined by $\Delta^{out}(\gamma)=\Delta\circ \gamma$, that is, we create a new $\Delta$ only at the output of a vertex.

The cohomology of the complex $(\kk[\Delta]\otimes\kk[u],\Delta\frac{\d}{\d {u}})$ is equal to $\kk$. Therefore,
the cohomology with respect to the differential $\Delta^{out}\frac{\d}{\d u}$ are generated by the graphs whose vertices are decorated by $u^0$ and there are no $\Delta$'s on the outputs of the vertices. This means that the whole graph is allowed to have only some $\Delta$'s at the global inputs of the graph. This kind of graphs span by definition the semi-direct composition $\Bar(s^2\Gerst^{\dual})\ltimes \kk[\Delta]$.
\end{proof}


\section{Main diagram of quasi-isomorphisms}
\label{sec::main_diagram}
In this section we present the full diagram of quasi-isomorphisms that connects $\Hycomm$ and $\BV/\Delta$.
We show how $\psi$-classes appears in the picture and how one can get an algebraic model of the Kimura-Stasheff-Voronov operad.
 
In the forthcoming Section~\ref{sec:diagrammatic} we are going to move through this diagram the generators $\mm_k$ of $\Hycomm$, 
and this way we obtain a quasi-isomorphism $\theta:\Hycomm\to\BV/\Delta$.

\begin{theorem}
\label{thm::diag:Hycom->BV} We have the following sequence of quasi-isomorphisms:
\begin{equation}
\xymatrix{
\left(\Bar(\kk[u]\otimes s^2\Gerst^{\dual}),\right. 
\left.\frac{\d}{\d{u}}\frac{\d}{\d w} +  \mu\right)
\ar@{->}[d]^{\varepsilon}
& &
*{\begin{array}{l}
\left(\Bar(\kk[u]\otimes s^2\Gerst^{\dual})\star \kk[\Delta]/\Delta,
\right.\\
{\ } \quad \left.\frac{\d}{\d{u}} \frac{\d}{\d{ w}} +  \mu + \Delta^{ad}\frac{\d}{\d u} + \Delta\frac{\d}{\d\phi}\right)
\end{array}}
\ar@{->}[ll]_-{j}
\ar@{->}[d]^{\epsilon}
\\
\left(\Bar(\Grav^{\dual}),\mu^{\Grav}\right)
\ar@{->}[d]^{\kappa}
& &
\left(
\frac{\Bar(s^2 \Gerst^{\dual})\ltimes \kk[\Delta]}{\Delta},
 \mu + \Delta\frac{\d}{\d\phi}
\right)
\ar@{->}[d]^{\pi}
\\
(\Hycomm,0)
\ar@{..>}[rr]^{\theta} & &
\left(\BV/\Delta, \Delta\frac{\d}{\d\phi} \right) 
}
\label{eq::diag::Hycom_BV}
\end{equation}
\end{theorem}

\begin{proof}
In Section~\ref{sec::operad::definitions} we give
a detailed description of all the morphisms involved in Diagram~\eqref{eq::diag::Hycom_BV} 
and prove that they are quasi-isomorphisms, except for $\theta$. Indeed,
\begin{itemize}
 \item[$\kappa$:]
The morphism $\kappa$ is a quasi-isomorphism because the operads $\Hycomm$ and $\Grav$ are Koszul dual to each other, 
see Section~\ref{sec::open:moduli}.
\item[$\varepsilon$:] 
The equivariant model of the operad $\Grav$ is 
discussed in Section~\ref{sec::grav}. We apply the cobar functor 
to the quasi-isomorphism $\varepsilon:\left(\kk[u]\otimes \Gerst^{\dual},\frac{\d}{\d u}\frac{\d}{\d  w}\right)\rightarrow \Grav^{\dual}$ 
described in Lemma~\ref{lem::gerst->grav}.
\item[$j$:] 
The morphism $j$ is a special case of the composition of the free product functor and the homotopy quotient functor discussed 
in Section~\ref{sec::homotopy_quotients_def}, see Corollary~\ref{lem::delta/delta}.
\item[$\epsilon$:] The existence of $\epsilon$ is discussed in Section~\ref{sec::BV=Gerst*Delta}.
The quasi-isomorphism property of $\epsilon$ is proved in Lemma~\ref{lem::uDelta->semiDelta} via a sequence of filtrations.
\item[$\pi$:] The map $\pi$ is obtained as a homotopy quotient of the quasi-isomorphism given by Equation~\eqref{eq::def_BGerst->BV}.
 The latter one is obtained from the standard Koszul resolution of $\Gerst$ 
(see Equations~\eqref{eq::BGerst->Gerst},\eqref{eq::pi::BGerst->Gerst}).
\item[$\theta:$] Section~\ref{sec:diagrammatic} contains a careful description of $\theta$
together with the proof of quasi-iso and commutativity of the diagram.
We take the generators of $\Hycomm$ and move the corresponding cocycles  through the diagram above in a clockwise direction.
We will show that the resulting map of generators from $\Hycomm$ to $\BV/\Delta$ defines a morphism of operads and
does not depend on particular choices of cocycles one should made in-between.
In particular, the image of the map $\theta$  coincides with 
the intersection of the kernel of differential $\Delta\frac{\d}{\d\phi}$  
with the suboperad of $\BV/\Delta$ generated by multiplication and $\phi_i$'s. 
\end{itemize}
\end{proof}

Recall from Section~\ref{sec::chern} that any given $S^1$-operad $\Q$ and a pair of natural numbers  $i<n$ 
defines an $S^1$-fibration over $\Q/\Delta(n)$ associated with the $S^1$-rotations in the $i$-th slot.
We will apply this construction for the operad $\BV$ in order to have another description 
of the line bundles over the moduli space $\oM_{0,n+1}$ formed by the cotangent lines at the marked point.
Recall that $\psi$-classes are the first Chern classes of these line bundles.
 Theorem~\ref{thm::psi-classes} below explains the algebraic counterpart of the action of $\psi$-classes in 
Diagram~\eqref{eq::diag::Hycom_BV}.
\begin{theorem}
\label{thm::psi-classes}
The $S^1$-fibration over the space of $n$-ary operation of the homotopy quotient by $S^1$ of the framed little discs operad
associated to the rotations in the $i$'th  slot 
coincides with the $S^1$-bundle over $\oM_{0,n+1}$ coming from the line bundle of the cotangent lines at the $i$-th marked point.

The algebraic models of the evaluation of the first Chern class of $S^1$-bundles under consideration
are underlined in the following refinement of commutative Diagram~\eqref{eq::diag::Hycom_BV}:
\begin{equation}
\label{eq::diag::psi_clas}
\xymatrix{
\left(\Bar(\kk[u]\otimes s^2\Gerst^{\dual}),
\frac{\d}{\d{u}}\frac{\d}{\d w} +  \mu\right)
\ar@{->}[d]^{\kappa\circ\varepsilon}
& &
*{\begin{array}{l}
\left(\Bar(\kk[u]\otimes s^2\Gerst^{\dual})\star \kk[\Delta]/\Delta,
\right.\\
{\ } \quad \left.\frac{\d}{\d{u}} \frac{\d}{\d{ w}} +  \mu + \Delta^{ad}\frac{\d}{\d u} + \Delta\frac{\d}{\d\phi}\right)
\end{array}}
\ar@{->}[ll]_-{j}
\ar@{->}[d]^{\pi\circ\epsilon}
\\
(\Hycomm,0)
\ar@{..>}[rr]^{\theta} & &
\left(\BV/\Delta, \Delta\frac{\d}{\d\phi} \right)
{\ar_{{\circ_i\frac{\d}{\d\phi_1}}}@/_1pc/(70,-23)*{};(78,-23)*{}}
{\ar_{{\psi_i}}@/_1pc/(-4,-23)*{};(4,-23)*{}}
{\ar^{{\circ_i\frac{\d}{\d u} + \circ_i\frac{\d}{\d \phi_1}}}@/^1pc/(70,8)*{};(78,8)*{}}
{\ar^{{\circ_i\frac{\d}{\d u}}}@/^1pc/(-4,5)*{};(4,5)*{}}
}
\end{equation}
Each operator drawn as a loop near the appropriate complex defines an operator which commutes with differential in this complex
and the vertical and horizontal arrows map these derivations one to another.
For example, the 
derivation $\circ_i\frac{\d}{\d u}$ is the differentiation by $u$-variable 
in the vertex attached to the $i$-th slot (input/output) 
of the element in the cobar complex $\Bar(\kk[u]\otimes s^2\Gerst^{\dual})$,
and
the differentiation $\circ_i\frac{\d}{\d \phi_1}$ means the noncommutative differentiation by $\phi_1$ in the algebra
$\kk\langle\phi_1,\phi_2,\ldots\rangle$ which is also attached to the $i$-th slot.
\end{theorem}
\begin{proof}
We omit the detailed proof of this Theorem because
the proof repeats the one of Theorem~\ref{thm::diag:Hycom->BV} 
and is based on the results of Getzler mentioned in Statement~\ref{stat::S1->BV->Grav}.
It is a direct check that the diagram commutes everywhere except the leftmost arrow.
From Proposition~\ref{prop::psi_algebraic} we know that the corresponding derivations drawn in the loops
represents the evaluation map with the first Chern class on the homology level.
Statement~\ref{stat::S1->BV->Grav} finishes the coincidence of the corresponding bundles and Chern classes respectively.
\end{proof}

Recall that for any operad $\Q$ with a chosen $S^1$-action we construct a functorial quasi-iso projection
$\eta_{\Q}: \left(\left(\Q /\Delta\right) \star \kk[\Delta], (\Delta-\Delta_{\Q})\frac{\d}{\d \phi}\right) \mapsto \Q$
(Compare with Equation~\eqref{eq::eta_Q}). 
We want to apply the functor $\eta_{\Q}\circ (\tt\star\kk[\Delta])$ to 
the main Diagram~\eqref{eq::diag::Hycom_BV}.
This operation is well defined because all operads involved in Diagram~\eqref{eq::diag::Hycom_BV}
are quasi-iso to the image of the functor of homotopy quotient by $\Delta$.
Moreover the composition of functors $\eta_{\Q}\circ (\tt\star\kk[\Delta])$ applied 
to the second column of Diagram~\eqref{eq::diag::Hycom_BV}
just removes the homotopy quotient.
On the other hand we show how this functor affects the differential 
if we apply the same functor to the left column of Diagram~\eqref{eq::diag::Hycom_BV}.
Indeed we have the following dg-model for the $\BV$-operad 
(the image of $\eta_{\Q}\circ (\tt\star\kk[\Delta])$ to the left-top operad from Diagram~\eqref{eq::diag::Hycom_BV}):
$$\left(\Bar(\kk[u]\otimes s^2\Gerst^{\dual})\star \kk[\Delta],
\frac{\d}{\d{u}}\frac{\d}{\d w} +  \mu + \Delta^{ad}\frac{\d}{\d u}\right).$$
Theorem~\ref{thm::psi-classes} says that the differential in the bottom of the column should replace 
the operator $\frac{\d}{\d u}$ by the evaluation of the corresponding $\psi$-class.
I.~e. the image of the bottom complex is $(\Hycomm\star\kk[\Delta], \Delta\psi)$ where the differential ``$\Delta\psi$''
is defined on the generators by the following formula:
\begin{equation*}
(\Delta\psi) \cdot \mm_n =  \sum_{i=0}^{n} (\psi_i \mm_n)\circ_i \Delta - 
\sum_{S_1\sqcup S_2= \{0,..,n\}} \mm_{|S_1|+1} \circ_{*} \Delta\circ_{*} \mm_{1+|S_2|}
\end{equation*}
The formulas have the same form whenever one uses the $\psi$-classes description of the $\Hycomm$-operad:
\begin{align*}
(\Delta\psi) \cdot 
\psi_0^{d_0}\ldots \psi_{i}^{d_i}\ldots\psi_n^{d_n} [\oM_{0,n+1}] = &
\sum_{i=0}^{n}  \psi_i \prod_{s=0}^{n} \psi_s^{d_s} [\oM_{0,n+1}]\circ_i\Delta  +
\\
& - \sum_{S_1\sqcup S_2= \{0,\ldots,n\}}  \prod_{s\in S_1} \psi_{s}^{d_s} [\oM_{0,|S_1|+1}] \otimes \Delta
\otimes \prod_{s\in S_2} \psi_{s}^{d_s} [\oM_{0,|S_2|+1}] 
\end{align*}

We finally ends up with the following corollary which seems to be quite useful in order to have a description of the 
Quillen homology and minimal resolution of $\BV$-operad:
\begin{corollary}
 There exists a commutative diagram of quasi-isomorphisms of operads:
$$
\xymatrix{
*{\begin{array}{l}
\left(\Bar(\kk[u]\otimes s^2\Gerst^{\dual})\star \kk[\Delta],\right. \\
{\ } \quad \left.\frac{\d}{\d{u}}\frac{\d}{\d w} +  \mu + \Delta^{ad}\frac{\d}{\d u}\right)
\end{array}}
\ar@{->}[d]^{\kappa\circ \varepsilon}
& &
\left(\Bar(\kk[u]\oplus \kk[u]\otimes s^2\Gerst^{\dual}),
\frac{\d}{\d{u}} \frac{\d}{\d{ w}} + \mu \right)
\ar@{->}[ll]_-{j}
\ar@{->}[d]^{\pi\circ\epsilon}
\\
(\Hycomm\star\kk[\Delta],\Delta\psi)
\ar@{..>}[rr]^{\theta} & &
(\BV,0) 
}
$$
\end{corollary}
Note that the operad $(\Hycomm\star \kk[\Delta], \Delta\psi)$ is an algebraic model of Kimura-Stasheff-Voronov 
operad (see e.g.\cite{Kimura} for details). 
Moreover, the map $\theta$ becomes an obviously defined projection that sends the operation $\mm_2\in\Hycomm(2)$ 
to the multiplication in $\BV$, 
$\Delta$ to $\Delta$ and all other generators $\mm_k$ for $k\geq 3$ of the operad $\Hycomm$ are mapped to $0$.

\section{Diagram chase}
\label{sec:diagrammatic}
This technical section consists of the precise description of the inverse maps that appear in Diagram~\eqref{eq::diag::Hycom_BV}.
The aim is to get precise formulas for the cocycles in this Diagram.
We move our cocycles through the Diagram~\eqref{eq::diag::Hycom_BV} step by step in the clockwise direction 
starting with the operad $\Hycomm$.

\subsection{The inverse of $\kappa$}
The generators of the cohomology of $\left(\Bar(\Grav^{\dual}),\mu^{\Grav}\right)$ that project under $\kappa$ to the cocycles $\mm_i$, $i=2,3,\dots$, are $\mmg_i$ described in Section~\ref{sec::open:moduli}, see Equation~\eqref{eq:mmg-first}.

\subsection{The inverse of $\varepsilon$}
\label{sec::inv::epsilon}
The complex $\left(\Bar(\kk[u]\otimes s^2\Gerst^{\dual}), \frac{\d}{\d{u}}\frac{\d}{\d w} +  \mu\right)$
has two differentials. The quasi-isomorphism $\varepsilon$ is the projection to the cohomology 
with respect to the differential $\frac{\d}{\d{u}}\frac{\d}{\d w}$.

Let us give the inductive procedure of writing an inverse map to $\varepsilon$. We will show how one can 
increase the number of inputs in order to write down a sequence of representing cocycles. The way we are doing 
that is not symmetric in the inputs; each cocycle will depend on the ordering of the inputs, but different 
orderings will give homologous cocycles. The map that increases the number of inputs is defined as a linear 
combination of some auxiliary maps that we introduce now.

Consider the natural embedding of the Orlik-Solomon algebras:
\begin{align*}
\iota_{I,n}\colon & \kk[u]\otimes \Gerst^{\dual}(I) \rightarrow 
\kk[u]\otimes \Gerst^{\dual}(I\sqcup\{n\}); 
\\ \notag
\iota_{I,n}\colon & w_{i j} \mapsto w_{i j}, \qquad 
\forall\ i,j\in I; 
\\ \notag
\iota_{I,n}\colon & u \mapsto u. 
\end{align*}
The meaning of this formula is the following. 
We just increase the number of inputs: the set of inputs $I$ is replaced by the set of inputs $I\sqcup \{n\}$.

We extend the map $\iota_{I,n}$ to a derivation of the bar-complex $\Bar(\kk[u]\otimes s^2\Gerst^{\dual})$. It is not well-defined for the operations of arity $\geq n$, because in this case it might appear that $I\ni n$. But we restrict the resulting map to the operations of arity $n-1$. We denote this extension by $\iota_{n}\colon \Bar(\kk[u]\otimes s^2\Gerst^{\dual})(n-1)\to \Bar(\kk[u]\otimes s^2\Gerst^{\dual})(n)$.

Now we define a collection of derivations $\varsigma_{s n}$,  $s=0,\ldots,n-1$, of the Bar-complex $\Bar(\kk[u]\otimes s^2\Gerst^{\dual})$. Again, this definition we need only in arity $(n-1)$, and it doesn't work in arity $\geq n$. The map $\varsigma_{s n}$ increases the set of inputs by the input $n$ in the same sense as $\iota_n$. Since $\varsigma_{s,n}$ is a derivation, it is enough to describe what happens when we apply it to a corolla $\gamma$. It produces a tree with one internal edge and two internal vertices.
One vertex coincides with the corolla $\gamma$ and the remaining vertex corresponds to a binary operation, that is, it has two inputs and one output.
There are two cases, $s=1,\dots,n-1$, and $s=0$. For $s=1$,\ldots,$n-1$ we have a map:
\begin{align*}
\varsigma_{s n}\colon & \kk[u]\otimes\Gerst^{\dual}(I)  \rightarrow 
\left( \kk[u]\otimes\Gerst^{\dual}(I\sqcup\{*\}\setminus \{s\})\right) \otimes 
\left( \kk[u]\otimes\Gerst^{\dual}(\{s,n\})\right); \\ \notag
 & \frac{u^{k}}{k!} f(w_{ij}) \mapsto \sum_{k_1+k_2=k} 
\frac{u^{k_1}}{k_1!} f(w_{ij})\otimes \frac{u^{k_2+1}}{(k_2+1)!} w_{s n}.
\end{align*}
Note that in the first factor on the right hand side we identify $w_{is}$ and $w_{i*}$ as it is prescribed by the cocomposition rules
defined in Equation~\eqref{eq::coordinates_in_BV}. For $s=0$ we have:
\begin{align*}
\varsigma_{0 n}\colon & \kk[u]\otimes\Gerst^{\dual}(I) \rightarrow  
\left( \kk[u]\otimes\Gerst^{\dual}(\{*,n\})\right)\otimes
\left( \kk[u]\otimes\Gerst^{\dual}(I)\right);  \\ \notag
 & \frac{u^{k}}{k!} f(w_{ij}) \mapsto - \sum_{k_1+k_2=k} 
 \frac{u^{k_2+1}}{(k_2+1)!} w_{* n} \otimes \frac{u^{k_1}}{k_1!} f(w_{ij}).
\end{align*}

\begin{lemma}
\label{lem::homotopy:Gerst->grav}
The map $\zeta_n:= \iota_n  + \sum_{s=0}^{n-1} \varsigma_{s n}$ is a chain map of homological degree $(-2)$ 
between the subcomplexes spanned by operations of arity $(n-1)$ and $n$:
\begin{equation*}
\zeta_n\colon \left(\Bar(\kk[u]\otimes s^2\Gerst^{\dual})(n-1),\frac{\d}{\d{u}}\frac{\d}{\d{w}} + \mu\right) 
\rightarrow
\left(\Bar(\kk[u]\otimes s^2\Gerst^{\dual})(n),\frac{\d}{\d{u}}\frac{\d}{\d{w}} + \mu\right)[-2] 
\end{equation*}
\end{lemma}

\begin{proof} The only thing that we have to check is that $\zeta_n$ commutes with the differential.
Since $\iota_n$ and $\varsigma_{s n}$, $s=0,\dots,n-1$, as well as $\frac{\d}{\d{u}}\frac{\d}{\d{w}}$ and $\mu$
are all derivations of the cobar complex, it is enough to check the compatibility on the generators.

First, observe that $[\iota_n,\frac{\d}{\d{u}}\frac{\d}{\d{w}}]=0$, because they does not interact with the $n$'th input. 
Then we compute the image of the commutator $[\iota_n,\mu]$ applied to the monomial $\frac{u^{k}}{k!} f(w_{ij})$, where the
indices $i,j$ belong a given set $K$:
\begin{align*}
& (\iota_n\mu - \mu\iota_n) \left( \frac{u^{k}}{k!} f(w_{ij})\right) =
\iota_n\left(\sum_{\begin{smallmatrix}I\sqcup J = K,\\ |J|\geq 2, |I|\geq 1,\\k_1+k_2=k\end{smallmatrix}} 
(-1)^{\deg_{w}{f^{I}}} \frac{u^{k_1}}{k_1!}f^{I}\otimes \frac{u^{k_2}}{k_2!}f^{J} \right) 
\\ \notag
& - 
\left(
\sum_{\begin{smallmatrix}(I\sqcup\{n\})\sqcup J = K\sqcup\{n\},\\ |J|\geq 2, |I|+1\geq 1,\\k_1+k_2=k\end{smallmatrix}}
 (-1)^{\deg_{w}{f^{I}}} \frac{u^{k_1}}{k_1!}f^{I}\otimes \frac{u^{k_2}}{k_2!}f^{J}
+
\sum_{\begin{smallmatrix}I\sqcup(J\sqcup\{n\}) = K\sqcup\{n\},\\ |J|+1\geq 2, |I|\geq 1,\\k_1+k_2=k\end{smallmatrix}}
 (-1)^{\deg_{w}{f^{I}}} \frac{u^{k_1}}{k_1!}f^{I}\otimes \frac{u^{k_2}}{k_2!}f^{J}
\right).
\end{align*}
Since $\iota_n$ increases the number of inputs in the operations but does not change the monomial, the only summands that are not canceled in the difference above are the ones with $|J|=1$ or $|I|=0$. Therefore,
\begin{equation*}
[\iota_n,\mu] \cdot \left( \frac{u^{k}}{k!} f(w_{ij}) \right) = - \sum_{k_1+k_2=k}
\left(\sum_{s\in K} (-1)^{\deg_{w}{f}} \frac{u^{k_1}}{k_1!}f^{*}\otimes \frac{u^{k_2}}{k_2!} 1^{s,n} 
+  \frac{u^{k_1}}{k_1!}1^{n*}\otimes \frac{u^{k_2}}{k_2!}f  
\right)
\end{equation*}
The monomial $f^{*}$ is obtained from $f$ by replacing the index $s$ by an additional index $*$
that appears in the cocomposition. 

Observe that $[\varsigma_{s n},\mu]=0$, $s=0,\dots,n-1$, since $\mu$ vanishes on binary operations. Meanwhile, for $s=1,\dots,n-1$ we have:
\begin{align*}
& \left[\varsigma_{s n},\frac{\d}{\d{u}}\frac{\d}{\d{w}}\right] \left( \frac{u^{k}}{k!} f(w_{ij})\right) 
\\ \notag
& = \sum_{k_1+k_2=k-1} \frac{u^{k_1}}{k_1!} \frac{\d f}{\d w} \otimes \frac{u^{k_2+1}}{(k_2+1)!} w^{sn} 
-\frac{\d}{\d u}\frac{\d}{\d w}
\left(\sum_{k_1+k_2 = k} \frac{u^{k_1}}{k_1!} f \otimes \frac{u^{k_2+1}}{(k_2+1)!} w^{sn}\right) 
\\ \notag
& = \sum_{k_1+k_2=k-1} \frac{u^{k_1}}{k_1!} \frac{\d f}{\d w} \otimes \frac{u^{k_2+1}}{(k_2+1)!} w^{sn}
-\sum_{k_1+k_2 =k}\left( \frac{u^{k_1-1}}{(k_1-1)!} \frac{\d f}{\d w} \otimes \frac{u^{k_2}}{k_2!} w^{sn}
+(-1)^{deg_{w}f -1}  \frac{u^{k_1}}{k_1!} f \otimes \frac{u^{k_2}}{k_2!} 1^{sn}\right) 
\\ \notag
& =  (-1)^{deg_w f} \sum_{k_1+k_2=k} \frac{u^{k_1}}{k_1!} f \otimes \frac{u^{k_2}}{k_2!} 1^{sn}.
\end{align*}
Here the sign $(-1)^{deg_{w}f -1}$ comes from the Koszul sign rule. Similarly, for $s=0$ we have:
\begin{equation*}
\left[\varsigma_{0 n},\frac{\d}{\d{u}}\frac{\d}{\d{w}}\right] \left( \frac{u^{k}}{k!} f(w_{ij})\right) =
  \sum_{k_1+k_2=k} \frac{u^{k_1}}{k_1!} 1^{* n} \otimes  \frac{u^{k_2}}{k_2!} f.
\end{equation*}

Finally, we see the cancellation:
\begin{align*}
& \left[\zeta_n, \frac{\d}{\d{u}}\frac{\d}{\d{w}} +\mu\right] \left( \frac{u^{k}}{k!} f(w_{ij})\right) 
\\ \notag
& = - \sum_{k_1+k_2=k}
\left(\sum_{s\in K} (-1)^{\deg_{w}{f}} \frac{u^{k_1}}{k_1!}f\otimes \frac{u^{k_2}}{k_2!} 1^{s,n} 
+  \frac{u^{k_1}}{k_1!}1^{n}\otimes \frac{u^{k_2}}{k_2!}f  
\right) 
\\ \notag
& \phantom{ = }\ + \sum_{k_1+k_2=k} \frac{u^{k_1}}{k_1!} 1^{* n} \otimes  \frac{u^{k_2}}{k_2!} f + 
\sum_{s=1}^{n-1}  \sum_{k_1+k_2=k} (-1)^{deg_w f} \frac{u^{k_1}}{k_1!} f \otimes \frac{u^{k_2}}{k_2!} 1^{sn} 
\\ \notag
& = 0.
\end{align*}
\end{proof}

We define a sequence of elements $\nu_n\in \Bar(\kk[u]\otimes s^2\Gerst^{\dual})(n)$, $n=2,3,\dots$. We set $\nu_2=w_{12}$ and define 
$\nu_{i+1}:=\zeta_{i+1}(\nu_i)$, $i=2,3,\dots$. Lemma~\ref{lem::homotopy:Gerst->grav} implies that 

\begin{corollary} The elements $\nu_n$ are the cocycles that project to the generators of the hypercommutative operad, $n=2,3,\dots$.
That is, for all $n\geq 2$ we have: 
\begin{equation*}
\left(\frac{\d}{\d{u}}\frac{\d}{\d{w}} +\mu\right)\nu_n=0 \quad \mbox{ and }\quad \kappa(\varepsilon(\nu_n))=\mm_n.
\end{equation*}
\end{corollary}

\begin{remark}
\label{rem::S_n::action::cocycles}
Any permutation $\sigma$ of the inputs will provide another choice of a cocycle given by
\begin{equation*}
 \zeta_{\sigma(n)}( \zeta_{\sigma({n-1})}(\ldots(\zeta_{\sigma(3)} (w_{\sigma(1)\sigma(2)}))\ldots)).
\end{equation*}
It is homologous to $\nu_n$ for any $\sigma\in S_n$.
\end{remark}

\subsubsection{The topological recursion relation}
\label{sec::TRR}
In this section we show how the formulas for $\nu_n$, $n=2,3,\dots$, imply the topological recursion relations.
\begin{lemma}
\label{lem::TRR}
 The following two cocycles are homologous:
\begin{equation*}
 \nu_{n} \circ_1 \frac{\partial}{\partial u} \quad \mbox{ and } \quad  
\sum_{|S_1\sqcup S_2| = n-2} \nu_{S_1\sqcup \{2,*\}} \otimes \nu_{S_2 \sqcup \{1\}}.
\end{equation*}
Similarly, the cocycle $\frac{\partial}{\partial u}\circ_0 \nu_{n}$ is homologous to the sum
$\sum_{|S_1\sqcup S_2| = n-2} \nu_{S_1\sqcup \{*\}} \otimes \nu_{S_2 \sqcup \{1,2\}}$.
\end{lemma}

\begin{proof}
Recall that the meaning of the derivation $\circ_i\frac{\partial}{\partial u}$ is to take the partial derivative with respect
to the variable $u$ attached to the $i$-th input (or, in the case of $i=0$, output) of the element in the cobar complex (c.~f. Theorem~\ref{thm::psi-classes}).
A direct computation similar to the one we made in the proof of Lemma~\ref{lem::homotopy:Gerst->grav} 
shows that the commutator $[\circ_i\frac{\d}{\d u},\zeta_n]$ acts on the monomial generator 
$\frac{u^{k}}{k!} f(w_{ij})$ by the following formula:
\begin{equation}
\left[\circ_i\frac{\d}{\d u},\zeta_n\right] \left( \frac{u^{k}}{k!} f(w_{ij}) \right) = 
\left[\circ_i\frac{\d}{\d u},\varsigma_{in} \right] \left( \frac{u^{k}}{k!} f(w_{ij}) \right) = \frac{u^{k}}{k!} f^{*}\otimes w_{in}.
\end{equation}
Here $f^{*}$ is obtained from $f$ by replacing the index $i$ with the index $*$ corresponding to the coproduct.

Note that two cocycles are homologous if and only if they have the same image under the morphism
$\kappa\circ \varepsilon$, since this morphism is a projection on the homology.
Recall that the augmentation map $\varepsilon$ annihilates all positive powers of $u$ and, in particular,
$\varepsilon\circ \varsigma_{sn}=0$. This implies the following sequence of identities:
\begin{align*}
\varepsilon\left(\nu_n\circ_1 \frac{\partial}{\partial u}\right) & = 
\varepsilon\left(\sum_{j=3}^{n}\zeta_n\ldots[\circ_1\frac{\d}{\d u},\zeta_j]\ldots\zeta_3 w_{12}\right)
\\ \notag
& = \varepsilon\left(\sum_{j=3}^{n}\iota_n\cdots\iota_{j+1}\left[\circ_1\frac{\partial}{\partial u},\varsigma_{1j}\right]
\iota_{j-1}\cdots \iota_{3}(w_{12})\right) 
\\ \notag
& = \varepsilon\left(\sum_{j=3}^{n} 
\sum_{S_1\sqcup S_2 = \{j+1,\ldots,n\}}\left( \left( \prod_{s\in S_1}\iota_{s} \right) \iota_{j-1}\cdots\iota_3 (w_{2 *})\right) \otimes 
\left( \prod_{s\in S_2}\iota_{s} (w_{1 j})\right)\right) 
\\ \notag
& =\varepsilon\left(\sum_{S_1\sqcup S_2 = \{3,\dots,n\}} \nu_{S_1\sqcup \{2,*\}} \otimes \nu_{S_2 \sqcup \{1\}}\right).
\end{align*}
The second statement of Lemma~\ref{lem::TRR} deals with the derivation $\frac{\d}{\d u}\circ_0$ 
with respect to the variable $u$ attached to the output. The proof is absolutely the same. 
\end{proof}

These homologous properties of the cocycles $\nu_n$ implies the topological recursion relations.
\begin{corollary}
\label{cor::TRR}
We have:
\begin{equation}
\label{eq::TRR}
\psi_0^{d_0}\psi_1^{d_1+1}\psi_2^{d_2}\cdots \psi_n^{d_n} [\oM_{0,n+1}] =
\sum_{S_1\sqcup S_2= \{3,\ldots,n\}}  \prod_{s\in S_1\sqcup\{0,2\}} \psi_{s}^{d_s} 
[\oM_{0,|S_1|+3}] 
\otimes \prod_{s\in S_2\sqcup\{1\}} \psi_{s}^{d_s} [\oM_{0,|S_2|+2}]. 
\end{equation}
Similarly,
\begin{equation}
\label{eq::TRR-0}
\psi_0^{d_0+1}\psi_1^{d_1}\cdots \psi_n^{d_n} [\oM_{0,n+1}] =
\sum_{S_1\sqcup S_2= \{3,\ldots,n\}}  \prod_{s\in S_1\sqcup\{0\}} \psi_{s}^{d_s} 
[\oM_{0,|S_1|+2}] 
\otimes \prod_{s\in S_2\sqcup\{1,2\}} \psi_{s}^{d_s} [\oM_{0,|S_2|+3}].
\end{equation}
\end{corollary}
\begin{proof} It follows from Theorem~\ref{thm::psi-classes} that we can use the partial derivation with respect to $u$ attached 
to the $i$-th input (respectively, to the output) instead of taking $\psi$-class in the $i$-th marked point (respectively, to the $0$-th marked point).
Therefore,
\begin{align*}
& \psi_1\prod_{s\in\{0,\ldots,n\}}\psi_s^{d_s} [\oM_{n+1}]  = 
\kappa\circ\varepsilon\left(\left(\frac{\partial}{\partial u}\circ_0\right)^{d_0}\prod_{s=1}^{n} \left(\circ_s \frac{\partial}{\partial u}\right)^{d_s} 
\circ_i\frac{\d}{\d u} \nu_n \right) 
\\ \notag &
= \kappa\circ\varepsilon\left(\left(\frac{\partial}{\partial u}\circ_0\right)^{d_0}\prod_{s=1}^{n} \left(\circ_s \frac{\partial}{\partial u}\right)^{d_s} 
\sum_{S_1\sqcup S_2 = \{3,\ldots,n\}}
\nu_{S_1\sqcup \{2,*\}} \otimes \nu_{S_2\sqcup\{1\}}\right)
\\ \notag &
= \sum_{S_1\sqcup S_2 = \{3,\ldots,n\}}
\kappa\circ\varepsilon\left(\left(\frac{\partial}{\partial u}\circ_0\right)^{d_0}\prod_{s\in S_1\sqcup\{2\}} 
\left(\circ_s \frac{\partial}{\partial u}\right)^{d_s} 
\nu_{S_1\sqcup \{2,*\}}\right) \otimes 
\kappa\circ\varepsilon\left(\prod_{s\in S_2\sqcup\{1\}} 
\left(\circ_s \frac{\partial}{\partial u}\right)^{d_s} 
\nu_{S_2\sqcup\{1\}}\right)
\\ \notag &
=
\prod_{s\in S_1\sqcup\{0,2\}} \psi_{s}^{d_s} 
[\oM_{0,|S_1|+3}] 
\otimes \prod_{s\in S_2\sqcup\{1\}} \psi_{s}^{d_s} [\oM_{0,|S_2|+2}] 
\end{align*}
The proof of the second statement of the corollary is exactly the same.
\end{proof}

\begin{remark}
The symmetric group acts on the cocycles $\nu_n$ changing them to the homologous one. 
Therefore, one can change the indices $1,2$ in the statement of Lemma~\ref{lem::TRR} and Corollary~\ref{cor::TRR} to any other pair of indices $i,j\in\{1,\ldots,n\}$. This completes our algebraic proof of the topological recursion relations.
\end{remark}

In particular, Equations~\eqref{eq::TRR} and~\eqref{eq::TRR-0} imply combinatorially that in the case $d_0+\cdots+d_n = n-2$
the product of $\psi$-classes evaluated on the fundamental class coincides 
with the iterated multiplication up to a multinomial coefficient:
\begin{equation}
\psi_0^{d_0}\ldots\psi_n^{d_n}[\oM_{n+1}](x_1,\ldots,x_n) = \frac{(n-2)!}{d_0!\ldots d_n!} m(x_1,\ldots,x_n)
\end{equation}
This formula explains the factors used in the definition of the map $\theta$ and, in particular, in Equation~\eqref{eq:psi-int}.

\subsection{The inverse of $j$}

In this section we construct the cocycles in the complex
\begin{equation}\label{eq:complex}
\left(\Bar(\kk[u]\otimes s^2\Gerst^{\dual})\star \kk[\Delta]/\Delta,
\frac{\d}{\d{u}} \frac{\d}{\d{ w}} +  \mu + \Delta\frac{\d}{\d u} + \Delta\frac{\d}{\d\phi}
\right)
\end{equation}
that represent there the generators $\mm_n$, $n=2,3,\ldots$, of $\Hycomm$.

The construction uses the definition of the homotopy quotient. 
Recall that the defining Equation~\eqref{eq::homotopy_quotients_def} implies the following two identities:
\begin{equation*}
(d+\Delta\frac{\d}{\d\phi}) \Phi(z) = \Phi(z)(d+ z\Delta), \quad
\Phi(z)^{-1} (d+\Delta\frac{\d}{\d\phi}) = (d+ z\Delta)\Phi(z)^{-1}
\end{equation*}
Therefore, the adjoint action of $\Phi$ on the complex~\eqref{eq:complex} 
given by $\Phi^{ad}(z)\colon \gamma \mapsto \Phi(z) \gamma \Phi(z)^{-1}$
satisfies the following equation:
\begin{equation}
\label{eq::Phi_adjoint}
\left(d+\Delta\frac{\d}{\d\phi}\right) \Phi^{ad}(z)(\gamma) = \Phi^{ad}(z)( d\gamma + z[\Delta,\gamma]).
\end{equation}
We use $\Phi(z)$ as a group-like element. This means that we want $\Phi^{ad}(z)$ must preserve the operadic 
composition, that is, $\Phi^{ad}(z)(\alpha\circ\beta)=(\Phi^{ad}(z)\alpha)\circ(\Phi^{ad}(z)\beta)$, where $z$ is an operator acting on corollas.

\begin{lemma}
\label{lem::inv::j}
Let $\nu$ be a cocycle in
$\left(\Bar(\kk[u]\otimes s^2\Gerst^{\dual}), \frac{\d}{\d{u}}\frac{\d}{\d{w}} +  \mu\right)$.
The cochain $\Phi^{ad}(\frac{\d}{\d u})\nu$ is a cocycle in 
the dg-operad $\left(\Bar(\kk[u]\otimes s^2\Gerst^{\dual})\star \kk[\Delta]/\Delta,
\frac{\d}{\d{u}} \frac{\d}{\d{ w}} +  \mu + \Delta^{ad}\frac{\d}{\d u} + \Delta\frac{\d}{\d\phi}
\right)$. Moreover, $j(\Phi^{ad}\left(\frac{\d}{\d u}\right)\nu) = \nu$.
\end{lemma}

\begin{proof}
Equation~\eqref{eq::Phi_adjoint} implies that
\begin{equation*}
\left(\frac{\d}{\d{u}} \frac{\d}{\d{ w}} +  \mu+\Delta\frac{\d}{\d \phi} + \Delta^{ad}\frac{\d}{\d u}\right) 
\Phi^{ad}\left(\frac{\d}{\d u}\right)\nu  
= 
\Phi^{ad}\left(\frac{\d}{\d u}\right) \left( \frac{\d}{\d{u}} \frac{\d}{\d{ w}} +  
\mu+ \Delta^{ad}\frac{\d}{\d u} -\Delta^{ad}\frac{\d}{\d u}\right)\nu 
= 0.
\end{equation*}
Since $j$ annihilates $\phi_i$, $i=1,2,\ldots$, the second statement of the lemma is obvious.
\end{proof}

Therefore, cocycles 
representing the generators $\mm_n$, $n=2,3,\ldots$, of $\Hycomm$ in the dg-operad~\eqref{eq:complex} 
can be given by the formula
\begin{equation}\label{eq:cocycles}
\Phi^{ad}\left(\frac{\d}{\d u}\right)\nu_n =
 \Phi^{ad}\left(\frac{\d}{\d u}\right)\zeta_{n}\cdots\zeta_{3}(w_{12}).
\end{equation}

\subsection{The projection $\pi\circ\epsilon$}
In this section we apply the projection $\pi\circ\epsilon$ to the cocycles given by Equation~\eqref{eq:cocycles}.

Recall that the projection $\epsilon$ from Section~\ref{sec::BV=Gerst*Delta} maps $u$ to $0$.
The projection $\pi$ given by Equations~\eqref{eq::def_BGerst->BV} and~\eqref{eq::pi::BGerst->Gerst} 
annihilates all non-binary trees in the cobar complex. In particular, $\pi$ vanishes on all contributions of the operators $\iota_n$ 
for the formulas $\nu_m$  for all $3\leq n\leq m$. Therefore,
\begin{equation}
\label{eq::Hycom_to_BV_with_u}
\theta_n:= \pi\circ\epsilon\left(\Phi^{ad}\left(\frac{\d}{\d u}\right)\cdot \nu_n\right) =
\pi\circ\epsilon\left(\Phi^{ad}\left(\frac{\d}{\d u}\right)
\sum_{\begin{smallmatrix}
       (i_3,\ldots,i_n) :\\
        \forall s \  0\leq i_s\leq s
      \end{smallmatrix}} 
\varsigma_{i_n n}\cdots\varsigma_{i_3 3} (w_{12}) \right).
\end{equation}

Finally we are able to state our main result:
\begin{theorem} \label{thm:theta}
 The map $\theta\colon\Hycomm\to\BV/\Delta$ defined by $\theta\colon \mm_n\mapsto \theta_n$ is a quasi-isomorphism of dg-operad. It makes 
the diagram~\eqref{eq::diag::Hycom_BV} commutative.
\end{theorem}

\begin{proof}
Theorem~\ref{thm::diag:Hycom->BV} implies that the cohomology of $\left(\BV/\Delta,\Delta\frac{\d}{\d\phi}\right)$ is isomorphic to $\Hycomm$.

We denote by $\Q\subset \BV/\Delta$ the intersection of the kernel of $\Delta\frac{\d}{\d\phi}$  
with the suboperad of $\BV/\Delta$ generated by multiplication and $\phi_i$'s. 
Observe that the suboperad $\Q\subset\BV/\Delta$ belongs to the cohomology. Indeed, by definition $\Q$ doesn't intersect the image of $\Delta\frac{\d}{\d\phi}$ and belongs to the kernel of $\Delta\frac{\d}{\d\phi}$.
Note that $\Delta$ does not appear in the representing cocycles $\Phi^{ad}(\frac{\d}{\d u}) \nu_n$
and, therefore, $\theta_n$ also does not contain $\Delta$ in its presentation in terms of the generators.
This implies that the cocycles $\theta_n$ belong to $\Q$, $n=2,3,\dots$. 

The same is true if we apply the diagram chase for any element of $\Hycomm$.
Therefore the full cohomology of $\left(\BV/\Delta,\Delta\frac{\d}{\d\phi}\right)$ is equal to $\Q$, and the map
$\mm_n\mapsto \theta_n$, $n\geq 2$, defines the isomorphism between $\Hycomm$ and $\Q$.
\end{proof}

We finish this section with a diagram that summarizes our chase of cocycles in Diagram~\eqref{eq::diag::Hycom_BV}:
\begin{equation*}
\xymatrix{
{
 \nu_n \in \Bar(\kk[u]\otimes s^2\Gerst^{\dual})
}
\ar@{->}[d]^{\kappa\circ\varepsilon}
&
{ 
\Phi^{ad}\left(\frac{\d}{\d u}\right)\nu_n
\in \frac{\Bar(\kk[u]\otimes s^2\Gerst^{\dual})\star \kk[\Delta]}{\Delta} 
}
\ar@{->}[l]_-{j}
\ar@{->}[d]^{\pi\circ\epsilon}
\\
{
 \mm_n \in
\Hycomm 
}
\ar@{..>}[r]^{\theta} & 
{
\theta_n
\in
\BV/\Delta 
}
}
\end{equation*}

\subsection{Examples for $n=2$ and $3$} 
\label{sec::computation:2:3}
In this section we compute Formula~\eqref{eq::Hycom_to_BV_with_u} for $n=2$ and $n=3$
and show the coincidence of two morphism (one via Givental graphs, another via diagram chase) for $n=2,3$. 
A direct computation for $n=2$ gives that 
\begin{equation*}
\theta_2=\pi\circ\epsilon\left( \Phi^{ad}\left(\frac{\d}{\d u}\right)
\left(w_{12}\right) \right)=\pi\left(w_{12}\right)=m_2,
\end{equation*}
which is exactly the formula for $\theta_2$ described in Section~\ref{sec:quasi}.

In the case of $n=3$, we have:
\begin{equation*}
\theta_3 = \pi\circ\epsilon\circ \Phi^{ad}\left(\frac{\d}{\d u}\right) \left(\varsigma_{03} (w_{12}) +  \varsigma_{13}(w_{12}) + \varsigma_{23}(w_{12})\right)
\end{equation*}
By definition,
\begin{equation*}
\varsigma_{03} (w_{12}) +  \varsigma_{13}(w_{12}) + \varsigma_{23}(w_{12})
= - (u w_{3*}) \circ_{*} w_{12} + w_{2*} \circ_{*} (u w_{13}) +  w_{1*} \circ_{*} (u w_{23}). 
\end{equation*}
Using that
\begin{align*}
\Phi^{ad}\left(\frac{\d}{\d u}\right) 
((u w_{3*}) \circ_{*} w_{12})
 = & 
\phi_1 \circ( w_{3*} \circ_{*} w_{12}) 
- (w_{3*}\circ_3 \phi_1) \circ_* w_{12} - w_{3*}\circ_{*}\phi_1\circ_* w_{12} \\ \notag
\Phi^{ad}\left(\frac{\d}{\d u}\right) 
(w_{2*} \circ_{*} (u w_{13}))
=&
w_{2*} \circ_{*} \phi_1\circ_{*} w_{13} - w_{2*}\circ_{*} w_{13}\circ_1 \phi_1 - w_{2*}\circ_{*} w_{13}\circ_3 \phi_1 \\ \notag
\Phi^{ad}\left(\frac{\d}{\d u}\right) 
(w_{1*} \circ_{*} (u w_{23}))
=&
w_{1*} \circ_{*} \phi_1\circ_{*} w_{23} - w_{1*}\circ_{*} w_{23}\circ_2 \phi_1 - w_{1*}\circ_{*} w_{23}\circ_3 \phi_1 
\end{align*}
it is then straightforward to compute the final expression for $\theta_3$
that appears to be a summation of $7$ terms and coincides with the formula for $\theta_3$ described in Example~\ref{ex::theta_23}.

The fact that we finally obtain the same formula for all $n\geq 0$ as in Section~\ref{sec:quasi} is based on Lemma~\ref{lem::TRR}
and in particular on the topological recursion relations considered in Theorem~\ref{thm::psi-classes}. An easier proof is given in
the next Section using a uniqueness argument. 

\subsection{Uniqueness}
\label{sec::uniqueness}
In order to get the coincidence of morphisms $\theta$ 
(first defined via summation of Givental graphs in Section~\ref{sec:quasi} 
and second via diagram chase in formula~\eqref{eq::Hycom_to_BV_with_u})
we just explain in the lemma below that there is no big freedom in the possible morphisms from $\Hycomm$ to $\BV/\Delta$.
\begin{proposition}
\label{thm::unique}
 Any graded automorphism of the operad $\Hycomm$ is defined by arbitrary dilations of $\mm_2$ and $\mm_3$.
I.~e. for a given pair $\lambda_2,\lambda_3$ there exist a unique automorphism of $\Hycomm$ given by formulas
$\mm_n \mapsto \lambda_2 \lambda_3^{n-2} \mm_n$ with $n\geq 2$;
moreover, any automorphism belongs to this system.
\end{proposition}
\begin{proof}
 Indeed, note that for all $n\geq 2$ the subspace of $\Hycomm(n)$ of homological degree $4-2n$ is onedimensional and is generated by 
the generator of $\Hycomm$ operad denoted earlier by $\mm_n$.
Therefore any graded automorphism of $\Hycomm$ should be of the form $\mm_n\mapsto \lambda_n \mm_n$.
The quadratic equations $\sum_{i+j}\mm_i\circ\mm_j =0$ in the operad $\Hycomm$ 
implies that the product $\lambda_i\lambda_j$ should depend only on the sum $i+j$.
By induction this follows that $\lambda_n = {\lambda_2}{\lambda_3}^{n-2}$.
\end{proof}
\begin{corollary}
The morphism $\theta:\Hycomm\rightarrow \BV/\Delta$ given by Formula~\eqref{eq::Hycom_to_BV_with_u} via summation over binary trees
coincides with the morphism $\theta$ described in Section~\ref{sec:quasi} via summation of Givental graphs.
\end{corollary}
\begin{proof}
In the proof of  Theorem~\ref{thm:theta} we explained that the suboperad $\Q\subset\BV/\Delta$ 
that is the intersection of the kernel of the differential $\Delta\frac{\d}{\d\phi}$ 
and the suboperad generated by multiplication and $\phi_i$'s is isomorphic to $\Hycomm$.
Two maps $\theta$ that we have constructed defines two particular (iso)morphisms from $\Hycomm$ to $\Q$.
We checked that this two morphisms coincide for $\mm_2$ and $\mm_3$. 
Therefore, our uniqueness Proposition~\ref{thm::unique} implies that they are the same for all $\mm_k$.
\end{proof}
\begin{remark}
 It is possible to show the coincidence of two formulas for $\theta$ without using uniqueness arguments.
The proof we know is technical and is based on the generalization of Lemma~\ref{lem::TRR}.
\end{remark}

\section{Givental theory}\label{sec:givental}

In this section prove Theorem~\ref{thm::formula_BV-Hycomm} using the Givental theory of a loop group action on the morphisms from $\Hycomm$ to an arbitrary operad. In fact, the action of the loop group on the $\Hycomm$-algebras has also a homological explanation. It comes from the action on trivializations of $\BV$-operator, and we explain this at the end of this section.

\subsection{Lie algebra action on morphisms of $\Hycomm$}

Consider an arbitrary operad $\P$. We consider morphisms of operads $\Hycomm\to \P$. We are going to introduce an infinitesimal action of the Lie algebra $\mathfrak{g}:=\P(1)\otimes \C[[z]]$ on space of morphisms, where $z$ is a formal variable and $\P(1)$ is considered as a Lie algebra with respect to the commutator $[x,y]=xy-yx$, $x,y\in \P(1)$.

In order to fix a morphism of $\Hycomm$ to $\P$, we consider a system of cohomology classes $\alpha_n\in H^{\udot}(\oM_{0,n+1},\C)\otimes \P(n)$. These classes must satisfy the following condition:
\begin{itemize}
\item For any map $\rho \colon \oM_{0,n_1+1}\times \oM_{0,n_2+1} \to \oM_{0,n+1}$, $n_1+n_2=n-1$, that realizes a boundary divisor in $\oM_{0,n+1}$ and induces the operadic composition $\circ_i\colon \Hycomm(n_1)\otimes \Hycomm(n_2)\to \Hycomm(n)$, we have:
\begin{equation}\label{eq:factorization}
\rho^*\alpha_n = \alpha_{n_1} \circledast_i \alpha_{n_2},
\end{equation}
where by $\circledast_i$ we denote the simultaneous product of cohomology classes and the $\circ_i$-composition in $\P$.
\end{itemize} 

The infinitesimal action of the Lie algebra $\mathfrak{g}$ is given by the explicit formulas. Consider an element $r_\ell z^\ell\in\mathfrak{g}$ for some $\ell\geq 0$. We have:
\begin{align} \label{eq:r-action}
r_\ell z^\ell. \alpha_n := 
& r_\ell \circ_1 \psi_0^\ell \alpha_n   + (-1)^{\ell+1} \sum_{m=1}^n \psi_m^\ell \alpha_n \circ_m r_\ell \\ \notag
& +\sum_{I\sqcup J=[n]} \sum_{i+j=\ell-1} (-1)^{i+1} \rho_* \left(\psi_1^i\alpha_{|I|+1} \circ_1 r_\ell \circledast_1 \psi_0^j\alpha_{|J|}\right)
\end{align}
Here in all cases $\circ_m$ denotes the operation in $\P$;
 $\psi_m$ denotes the $\psi$-class in the corresponding moduli space 
($\oM_{0,n+1}$ in the second summand or $\oM_{0,|I|+2}$ and $\oM_{0,|J|+1}$ in the third summand), 
that is, the first Chern class of the line bundle with the fiber $T^*_{x_m} C$ over the curve $(C,x_0,x_1,\dots,x_k)\in \oM_{0,k+1}$ ($k$ is then equal to $n$, $|I|$, and $|J|+1$ respectively). Moreover, we always assume that the ``output'' marked point is $x_0$, and, in the third summand, we assume that the map $\rho$ attaches the output point of $\oM_{0,|J|+1}$ to the first input (that is, the point $x_1$) of $\oM_{0,|I|+2}$.

\begin{example} In the case $\ell=0$ we simply have $r_0z^0.\alpha_n=[r_0,\alpha_n]$ in the sense of commutation of operadic compositions in $\P$.
\end{example}

The formula for the $\mathfrak{g}$-action is a generalization of the formulas considered in~\cite{Giv3,Lee1,Sha,Tel}, and we refer the reader to these papers for a more detailed introduction to the Givental theory.

\begin{lemma}\label{lem:infinitesimal} For any $r=\sum_{\ell=0}^\infty r_\ell z^\ell\in\mathfrak{g}$ 
and any system of classes $\alpha_n\in H^{\udot}(\oM_{0,n+1},\C)\otimes \P(n)$, $n\geq 2$, 
that satisfies the factorization condition~\eqref{eq:factorization}, 
the classes $\alpha_n+\epsilon \cdot r.\alpha_n\in H^{\udot}(\oM_{0,n+1},\C)\otimes \P(n)$ also satisfy the factorization condition~\eqref{eq:factorization} in the first order in $\epsilon$.
\end{lemma}

\begin{proof} It is a straightforward generalization of Proposition 6.9 in~\cite{Tel}.
\end{proof}

It follows from Lemma~\ref{lem:infinitesimal} that for any morphism $g\colon\Hycomm\to \P$ and an arbitrary sequence of elements $r_\ell\in \P(1)$, $\ell=1,2,\dots$, we obtain a new morphism $\exp(r.)g\colon \Hycomm\to \P$, $r=\sum_{\ell=1}^\infty r_\ell z^\ell$, by exponentiation of the infinitesimal Lie algebra action defined above. This means that we define an action of the Lie group
$G=\{M(z) \in O(1)\otimes \C[[z]], M(0)=1\}$ on the space of morphisms $\Hycomm\to \P$.

\subsection{Application to the $\BV$-operad}

We consider the morphism $\theta_0\colon \Hycomm\to \BV$ 
that sends the generator $\mm_k$ to the iterated multiplication $m()$, $k\geq 2$. 
In terms of the infinitesimal Givental action the condition that $\Delta$ is the second order operator 
with respect to the multiplication can be written as 
\begin{equation}\label{eq:BV}
(\Delta z^1).\theta_0 = 0
\end{equation}
(it is proved in a bit different terms in~\cite[Proposition 1]{Sha}).

The same map $\theta_0$ can be also considered as a map to $\BV/\Delta$. In this case, in addition to Equation~\eqref{eq:BV} we also have 
\begin{equation}\label{eq:d}
\left(\Delta\frac{\d}{\d \phi} z^0\right).\theta_0 = 0.
\end{equation}
(abusing a little bit the notation we think of $\Delta \frac{\d }{\d \phi}$ as an element of $\BV/\Delta$ 
such that the differential is given by the commutator with this element).

Consider the map $\theta\colon \Hycomm\to \BV/\Delta$ 
defined by $\exp(\phi(z).)\theta_0$. 
There are several observations. 
First of all, just by construction, $\theta$ is a morphism of operads. Second, we want to show that $\theta$ is a morphism of dg-operads, that is, 
$(\Delta\frac{\d}{\d\phi} z^0).\theta = 0$. This follows from the following computation:
\begin{equation*}
\left(\Delta\frac{\d}{\d\phi} z^0\right).\theta = \left(\Delta\frac{\d}{\d\phi}z^0\right).\exp(\phi(z).)\theta_0 
= \exp(\phi(z).)\left(\Delta\frac{\d}{\d\phi}z^0+\Delta z^1\right). \theta_0 = 0.
\end{equation*}
Here the first equality is the definition of $\theta$, the second one is a consequence of Equation~\eqref{eq:d-delta},
 and the third equality follows from Equations~\eqref{eq:BV} and~\eqref{eq:d}.

Thus we see that  $\theta(\Hycomm)\subset \Q\subset \BV/\Delta$, 
where $\Q$ is the suboperad considered in the proof of Theorem~\ref{thm:theta}, that is, $\Q$ is the intersection of the kernel of $\Delta\frac{\d}{\d \phi}$ with the suboperad generated by the multiplication and $\phi_i$'s, $i=1,2,\ldots$. 

In the proof of Theorem~\ref{thm:theta} we observed that $\Q$ is isomorphic to $\Hycomm$. Moreover, a simple degree count shows that the map $\theta\colon \mm_k\mapsto \theta_k$, $k=2,3,\dots$, preserves the degrees. Therefore, $\theta$ maps generators to generators, and it is an isomorphism between $\Hycomm$ and $\Q$.

The last observation is that $\theta$ is exactly the map constructed in Section~\ref{sec:quasi} in terms of graphs. 
This can be observed by an explicit exponentiation of the formula~\eqref{eq:r-action}, 
and, for example, it is also explained in~\cite[Section 6.14]{Tel} and~\cite{DunShaSpi}. 
This completes the proof of Theorem~\ref{thm::formula_BV-Hycomm}.

\subsection{Homological origin of the Givental action}

In this section we explain how the Givental group action emerges naturally via the loop
group action on trivializations of $\Delta$.

Consider a finite-dimensional $\Hycomm$-algebra $V$ with zero differential. Let $\bar V$ be the corresponding differential 
graded $\BV$-algebra with the differential $d$, and we denote by $\phi_i$ the corresponding additional operators coming from structure of $\BV/\Delta$
on $\bar V$.

Consider an arbitrary sequence of endomorphisms $\alpha_i\in \End(V)$.
Since the cohomology of $\bar V$ coincides with $V$, we can define a sequence of 
endomorphisms $\bar{\alpha}_i\in\End(\bar V)$ such that they commute with the differential on $\bar V$ 
and their restrictions to the cohomology coincide with $\alpha_i$, $i=1,2,\ldots$.

We have:
\begin{equation*}
\exp\left(-\sum_{i=1}^\infty\bar{\alpha}_i z^{i}\right) d \exp\left(\sum_{i=1}^\infty\bar{\alpha}_i z^i\right) = d 
\end{equation*}
Therefore, 
\begin{equation*}
\exp(-\phi(z))\exp\left(-\sum_{i=1}^\infty\bar{\alpha}_i z^{i}\right) d \exp\left(\sum_{i=1}^\infty\bar{\alpha}_i z^i\right)\exp(\phi(z))=
d+z\Delta
\end{equation*}

Thus we see that the sequence of operators $\phi'_i$ given by the formula
\begin{equation*}
\phi'(z)=\sum\phi'_i z^{i} \colon = \ln(\exp(\bar{\alpha}(z)\exp(\phi(z)))
\end{equation*}
defines a new $\BV/\Delta$-algebra structure on $(\varsigma(V),d)$. This structure induces a new $\Hycomm$-algebra structure on $V=H^\bullet(\bar V, d)$. 

\begin{theorem}
The new $\Hycomm$-algebra structure on $V$ coincides with the one obtained by the Givental group action of the element $\exp\left(\sum_{i=1}^\infty\alpha_iz^i\right)$ applied to the original $\Hycomm$-algebra.
\end{theorem}

\begin{proof} It is easier to compare the infinitesimal deformations. Indeed, assume that $\sum_{i=1}^\infty\bar{\alpha}_i z^i= r_\ell z^\ell$ and we consider the first order deformation in $r_\ell$. In this case $\phi'(z)=\phi(z)+ r_\ell z^\ell$. Then it is just a tautological observation to see that the corresponding deformation of the formulas for $\theta_k$ in Section~\ref{sec:quasi}, $k\geq 2$,  is given by Equation~\eqref{eq:r-action}.
\end{proof}

\end{document}